\documentclass[final]{siamltex}
\usepackage{amsmath} 
\usepackage{amssymb}  
\usepackage{color}
\usepackage{graphicx}
\usepackage{hyperref}
\usepackage[ruled,vlined]{algorithm2e}
\usepackage{psfrag}
\usepackage{comment}
\newtheorem{remark}[theorem]{Remark}

\DeclareMathOperator{\sing}{sing}

\DeclareMathOperator{\hi}{hi}
\DeclareMathOperator{\inst}{inst}

\newcommand {\mat}      [1] {\left[\begin{array}{#1}}
\newcommand {\rix}          {\end{array}\right]}
\newcommand{\beq}{\begin{equation}}
\newcommand{\eeq}{\end{equation}}
\newcommand{\od}{\mathrm{od}}
\newcommand{\ev}{\mathrm{ev}}

\newcommand{\cM}{{\mathcal M}}
\newcommand{\G}{{\mathbf G}}
\newcommand{\M}{{\mathbf M}}
\newcommand{\V}{{\mathcal V}}

\newcommand{\R}{\mathbb R}

\newcommand{\Sym}{\mathrm{Sym}}
\newcommand{\Skew}{\mathrm{Skew}}

\def\iu{\imagunit}

\def\i{{\rm i}}
\def\r{{\rm r}}
\def\eps{\varepsilon}

\def\l{\lambda}

\def\tp{{\rm T}}

\newcommand{\Z}{{\mathbf Z}}

\renewcommand{\Re}{{\mbox{\rm Re}}}
\renewcommand{\Im}{{\mbox{\rm Im}}}

\newcommand{\imagunit}{{\bf i}}

\newcommand{\conj}[1]{\overline{#1}}

\newcommand{\epstar}{\eps^{\ast}}

\newcommand{\bmx}{\begin{pmatrix}}
\newcommand{\emx}{\end{pmatrix}}
\newcommand{\bsm}{\left[\begin{smallmatrix}}
\newcommand{\esm}{\end{smallmatrix}\right]}
\newcommand{\norm}[1]{\|#1\|}

\newcommand {\eproof}
      {\space
        {\ \vbox{\hrule\hbox{\vrule height1.3ex\hskip0.8ex\vrule}\hrule}}
        \par}
\title{Computation of the nearest structured matrix triplet with common null space}

\author{Nicola Guglielmi\footnotemark[1] \and Volker Mehrmann\footnotemark[3]}

\begin{document}

\maketitle

\renewcommand{\thefootnote}{\fnsymbol{footnote}}
\footnotetext[1]{Gran Sasso Science Institute,
Section of Mathematics,
Via Crispi 7,
I-67010    L' Aquila,  Italy. Email: {\tt nicola.guglielmi@gssi.it}}
\footnotetext[2]{Institut f\"ur Mathematik, MA 4-5, TU Berlin,
       Str. der 17 Juni 136,
       D--10623 Berlin,
       Germany. Email: {\tt mehrmann@math.tu-berlin.de}}
\renewcommand{\thefootnote}{\arabic{footnote}}

\begin{abstract}
We study computational methods for computing the distance to  singularity, the
distance to the nearest high index problem, and the distance to instability for linear differential-algebraic systems (DAEs) with dissipative Hamiltonian structure.  While for general unstructured DAEs the characterization of these distances is very difficult, and partially open, it has been shown in \cite{MehMW20} that for dissipative Hamiltonian systems and related matrix pencils there exist explicit characterizations. We will use these characterizations for the development of computational methods to compute these distances via methods that follow the flow of a differential equation converging to the smallest perturbation that destroys the property of regularity, index one or stability.
\end{abstract}

\noindent


\section{Introduction}

We derive computational methods for determining the distance to  singularity, the
distance to the nearest high index problem, and the distance to instability for \emph{linear, time-invariant differential-algebraic systems (DAEs) with dissipative Hamiltonian structure (dHDAEs)}. Such systems arise as linearization of general dHDAEs along a stationary solution and have the form
\begin{equation}\label{dH}
E\dot x = (J-R)x+f,
\end{equation}
with constant coefficient matrices  $E,J, R\in{\mathbb R}^{n,n}$, $J=-J^\top$, and $E=E^\top,R=R^\top$ symmetric positive semidefinite, a differentiable state function $x:{\mathbb R} \to {\mathbb R}^n$ and a right hand side $f:{\mathbb R} \to {\mathbb R}^n$, see
\cite{BeaMXZ18,GilMS18,JacZ12,MehMW18,MehMW20,MehM19,Sch13,SchJ14,SchM02,SchM18} for slightly varying definitions and a detailed analysis of such systems also in the context of the more general \emph{port-Hamiltonian systems}.  The matrix $E$ is associated with the Hessian of the associated \emph{Hamiltonian} energy function, which in the quadratic case has the form  $\mathcal H(x)= \frac 12 x^\top E x $. It is well-known \cite{BeaMXZ18,MehM19,SchJ14} that pHDAEs satisfy a \emph{dissipation inequality}
 $ {\mathcal H}\big(x(t_1)\big)-{\mathcal H}\big(x(t_0)\big) \leq 0$ for $t_1\geq t_0$.

Such pHDAE systems arise in all areas of science and engineering
\cite{AltMU20,BeaMXZ18,EggKLMM18,MehM19,SchJ14} as linearizations, space discretization, or approximation of physical systems and are usually  model descriptions with uncertainties. It is therefore important to know whether the model is close to an ill-posed or badly formulated model, and this has been an important research topic recently, see \cite{AchAM21,AliMM20,BeaMV19,GilMS18,GilS17,GilS18,MehMS16,MehMS17,MehMW20,MehV20}. Since the  system properties of \eqref{dH} are characterized by investigating the corresponding \emph{dissipative Hamiltonian (dH) matrix pencil}
\begin{equation}\label{dHwoQ}
L(\lambda):=\lambda E-(J-R),
\end{equation}
the discussed nearness problems can be characterized by determining the distance to the nearest singular pencil, i.e., a pencil with a $\det(\lambda E-J+R)$ identically zero, the distance to the nearest high-index problem, i.e., a problem with Jordan blocks associated to the eigenvalue $\infty$ of size
bigger than one, or the nearest problem on the boundary of the unstable region, i.e. a problem with purely imaginary eigenvalues. To compute these  distances is very difficult for general linear systems \cite{BerGTWW17,BerGTWW19,ByeHM98,GilMS18,GilS17,GilS18,GugLM17,MehMW15}. However, if one restricts the perturbations to be structured, i.e. one considers  \emph{structured distances} within the class of linear time-invariant dHDAEs, the situation changes completely, see  \cite{GilMS18,GilS17,GilS18,MehMW18,MehMW20}, and one obtains very elegant characterizations that can be used in numerical methods to compute these distances.

These methods are usually based on  non-convex optimization approaches. In contrast to such approaches, we derive computational methods to compute these structured distances by following the flow of a differential equation. This approach  has been shown to be extremely effective for computing the distance to singularity for general matrix pencils \cite{GugLM17} and we will show that this holds even more so in the structured case.

Neither the methods based on non-convex optimization nor the methods based on following a flow for general pencils or structured pencils are really feasible for large scale problems.
To treat the large sparse case they have to be combined with projections on the sparsity structure and model reduction methods, see \cite{AliBMSV17,AliMM20}, which  intertwine the optimization step with model reduction via interpolation. Here we discuss only the small scale case, but the combination with interpolation methods can be carried out in an analogous way as in \cite{AliMM20}.

The paper is organized as follows.
In Section~\ref{sec:prelim} we recall a few basic results about  linear time-invariant dHDAE systems.
In Section~\ref{sec:optB} we discuss optimization methods that are based on gradient flow computations. Since the cases of even and odd dimension are substantially different, in Section~\ref{sec:gradnodd} we specialize these methods for the optimization problems associated with the three discussed distance problems for the case that the state dimension is odd, while in Section~\ref{sec:even} we discuss the case that the state dimension is even.
Since it is known that the optimal perturbations are rank two matrices, in Section~\ref{sec:rank2} for the odd size case we discuss the special situation that we restrict the perturbation to be at most of rank two. In Section~\ref{sec:eps} we briefly discuss the iterative procedure for computing the optimal $\eps$ in the upper level of the two level procedure. In all cases, we present numerical examples.

\section{Preliminaries}\label{sec:prelim}
We use the following notation. The set of symmetric (positive semidefinite) matrices in $\mathbb R^{n,n}$ is denoted by $\Sym^{n,n}$ ($\Sym^{n,n}_{\geq0}$), and the skew-symmetric  matrices in $\mathbb R^{n,n}$ by $\Skew^{n,n}$. By $\| X\|_F$ we denote the Frobenius norm
of a (possibly rectangular) matrix $X$, we extend this norm to matrix tuples $\mathcal X=(X_0,\dots,X_k)$ via
 $\| \mathcal X\|_F=\|[X_0,\dots,X_k]\|_F$. For $A, B \in {\mathbb C}^{n,n}$, we denote by
\[
\langle A,B\rangle = \text{tr}(B^H A)
\]
the Frobenius inner product on ${\mathbb C}^{n,n}$, where $B^H$ is the conjugate transpose of $B$. The Euclidian norm in $\mathbb R^n$ is denoted by $\|~\|$. By $\lambda_{\min}(X)$ we denote the smallest eigenvalue of $X\in \Sym^{n,n}_{\geq0} $. The real and imaginary part of a complex matrix $A\in {\mathbb C}^{n,n}$ is denoted by $\Re(A)$, $\Im(A)$, respectively.

To characterize the properties of dHDAEs of the form \eqref{dH}, we exploit the \emph{Kronecker canonical form} of the associated matrix pencil \eqref{dHwoQ}, see \cite{Gan59a}. If $\mathcal J_n(\lambda_0)$ denotes the standard upper triangular Jordan block of size $n\times n$ associated with an eigenvalue $\lambda_0$ and $\mathcal L_n$ denotes the standard right Kronecker block of size $n\times(n+1)$, i.e.,
\[
\mathcal L_n=\lambda\left[\begin{array}{cccc}
1&0\\&\ddots&\ddots\\&&1&0
\end{array}\right]-\left[\begin{array}{cccc}
0&1\\&\ddots&\ddots\\&&0&1
\end{array}\right],
\]
then for $E,A\in {\mathbb C}^{n,m}$ there exist nonsingular matrices
$S\in {\mathbb C}^{n,n}$ and $T\in {\mathbb C}^{m,m}$ that transform the pencil to \emph{Kronecker canonical form},
\begin{equation}\label{kcf}
S(\lambda E-A)T=\diag({\cal L}_{\epsilon_1},\ldots,{\cal L}_{\epsilon_p},
{\cal L}^\top_{\eta_1},\ldots,{\cal L}^\top_{\eta_q},
{\cal J}_{\rho_1}^{\lambda_1},\ldots,{\cal J}_{\rho_r}^{\lambda_r},{\cal N}_{\sigma_1},\ldots,
{\cal N}_{\sigma_s}),
\end{equation}
where $p,q,r,s,\epsilon_1,\dots,\epsilon_p,\eta_1,\dots,\eta_q,\rho_1,\dots,\rho_r,\sigma_1,\dots,\sigma_s\in\mathbb N_0$ and
$\lambda_1,\dots,\lambda_r\in\mathbb C$, as well as ${\cal J}_{\rho_i}^{\lambda_i}=I_{\rho_i}-\mathcal J_{\rho_i}(\lambda_i)$
for $i=1,\dots,r$ and $\mathcal N_{\sigma_j}=\mathcal J_{\sigma_j}(0)-I_{\sigma_j}$ for $j=1,\dots,s$.

For real matrices and real transformation matrices $S,T$, the  blocks ${\cal J}_{\rho_j}^{\lambda_j}$ with $\lambda_j\in\mathbb C\setminus\mathbb R$ are in \emph{real Jordan canonical form} associated to the corresponding pair of conjugate complex eigenvalues,  the other blocks are the same. A real or complex eigenvalue is called \emph{semisimple} if the largest associated Jordan block in the complex Jordan form has size one and the sizes $\eta_j$ and $\epsilon_i$
are called the \emph{left and right minimal indices} of $\lambda E-A$, respectively. A pencil $\lambda E-A$,  is called \emph{regular} if $n=m$ and
$\det(\lambda_0 E-A)\neq 0$ for some $\lambda_0 \in \mathbb C$,
otherwise it is called \emph{singular}; $\lambda_1,\dots,\lambda_r\in\mathbb C$ are called the finite eigenvalues of $\lambda E-A$, and $\lambda_0=\infty$ is an eigenvalue of $\lambda E-A$ if zero is an eigenvalue of
the $\lambda A-E$.
The size of the largest block ${\cal N}_{\sigma_j}$ is
called the \emph{index} $\nu$ of the pencil $\lambda E-A$.

The definition of stability for differential-algebraic systems varies in the literature. We call a pencil $\lambda E-A$  {\em Lyapunov stable (asymptotically stable)} if it is regular, all finite eigenvalues are in the closed (open) left half plane, and the ones lying on the imaginary axis (including $\infty$) are semisimple \cite{DuLM13}.
Note that pencils with eigenvalues on the imaginary axis or at $\infty$ are on the boundary of the set of asymptotically systems and those with multiple, but semisimple, purely imaginary eigenvalues (including $\infty$) lie on the boundary of the set of Lyapunov stable pencils.

The following theorem summarizes some results of \cite{MehMW18,MehMW20} for real dH pencils; note that some of the results also hold in the complex case.
\begin{theorem}\label{thm:singind}
Let $E,R\in \mathbb R^{n,n}$ be symmetric and positive semidefinite, and
$J=-J^\top\in\mathbb R^{n,n}$. Then the following
statements hold for the pencil $L(\lambda)=\lambda E-J+R$.
\begin{enumerate}
\item[\rm (i)] If $\lambda_0\in\mathbb C$ is an eigenvalue of $L(\lambda)$ then $\operatorname{Re}(\lambda_0)\leq 0$.
\item[\rm (ii)] If $\omega\in\mathbb R$ and $\lambda_0=i\omega$ is an eigenvalue of $L(\lambda)$, then
$\lambda_0$ is semisimple. Moreover, if the columns of $V\in\mathbb C^{m,k}$ form a basis of a regular deflating
subspace of $L(\lambda)$ associated with $\lambda_0$, then $RQV=0$.
\item[\rm (iii)] The index of $L(\lambda)$ is at most two.
\item[\rm (iv)] All right and left minimal indices of $L(\lambda)$ are zero (if there are any).
\item[\rm (v)] The pencil $L(\lambda)$ is singular if and only if $\,\ker J\cap\ker E\cap\ker R\neq\{0\}$.
\end{enumerate}
\end{theorem}
Based on Theorem~\ref{thm:singind}, in \cite{MehMW20} the following distance problems were introduced for dH pencils.
%
\begin{definition}\label{def:distances}
Let $\mathcal{L}$ denote the class of square $n\times n$ real matrix pencils of the form \eqref{dHwoQ}.
Then
\begin{enumerate}
\item  the  \emph{structured distance to singularity}  is defined as
\begin{equation}\label{distEJR}
d_{\sing}^{\mathcal L}\big(L(\lambda)):=\inf\big\{\big\|  \Delta_L(\lambda) \big\|_F\ \big|\
L(\lambda)+\Delta_L(\lambda)\in\mathcal{L} \mbox{ and is singular}\big\};
\end{equation}
\item the  \emph{structured distance to the nearest high-index problem} is defined as
\begin{equation}\label{indexdistEJR}
d_{\hi}^{\mathcal{L}}\big(L(\lambda)):=\inf\big\{
\big\|  \Delta_L(\lambda)
\big\|_F\ \big| \
L(\lambda)+\Delta_L(\lambda)\in \mathcal{L} \mbox{ and is of index}\geq 2\big\};
\end{equation}
\item the  \emph{structured distance to instability} is defined as
\begin{equation}\label{instdistEJR}
d_{\inst}^{\mathcal{L}}\big(L(\lambda)\big):=\inf\big\{
\big\|
\Delta_L(\lambda)
\big\|_F \
\big| \ L(\lambda)+\Delta_L(\lambda)\in\mathcal{L} \mbox{ and is unstable}
\big \}.
\end{equation}
Here $\Delta_L(\lambda)=\lambda \Delta_E -\Delta_J + \Delta_R$, with $ \Delta_J\in \Skew^{n,n}$, $E+\Delta_E,R+\Delta_R\in \Sym^{n,n}_{\geq0}$, and
$\big\| [ \Delta_J, \Delta_R, \Delta_E]\big\|_F=\big\| [ \Delta_L(\lambda)] \big\|_F$.
\end{enumerate}
\end{definition}
It has also been shown in \cite{MehMW20} that these distances can be characterized as follows.
\begin{theorem}\label{singind}
Let  $\lambda E -J+R \in \mathcal L$.
Then the following statements hold.
\begin{enumerate}
\item\label{singindI}
Define for a matrix $Y\in {\mathbb R}^{n,n}$, the matrix $\Delta_Y^u=-uu^\top Y-Yuu^\top+uu^\top Yuu^\top$. The distance to singularity \eqref{distEJR} is attained with a perturbation
$\Delta_E=\Delta_E^u$, $\Delta_J=\Delta_J^u$, and $\Delta_R=\Delta_R^u$
for some $u\in\mathbb R^n$ with $\|u\|_2=1$. It is given by
\begin{small}
\begin{multline*}
d_{\sing}^{\mathcal{L}}\big(\lambda E - J+R\big)\\
=\min_{u \in{\mathbb R}^n\atop\norm u=1} \sqrt{2\norm{Ju}^2+2\big\|(I-uu^\top)Eu\big\|^2+(u^\top Eu)^2+
2\big\|(I-uu^\top)Ru\big\|^2+(u^\top Ru)^2}
\end{multline*}
\end{small}
and is bounded as
\begin{equation}\label{lmino}
\sqrt{\lambda_{\min} (-J^2+R^2+E^2)}\leq d_{\sing}^{\mathcal{L}}\big(\lambda E - J+R\big)\leq  \sqrt{2\cdot\lambda_{\min}(-J^2+R^2+E^2)}.
\end{equation}
\item\label{singindII} The structured distance to higher index~\eqref{indexdistEJR} and the structured distance to instability~\eqref{instdistEJR} coincide and satisfy
\begin{multline*}
d_{\hi}^{\mathcal{L}}\big(\lambda E - J+R\big)=d_{\inst}^{\mathcal{L}}\big(\lambda E - J+R\big)\\
=\min_{u \in{\mathbb R}^n\atop\norm u=1} \sqrt{2\big\|(I-uu^\top)Eu\big\|^2+(u^\top Eu)^2+
2\big\|(I-uu^\top)Ru\big\|^2+(u^\top Ru)^2} 
\end{multline*}
and are bounded as
\begin{equation}
\sqrt{\lambda_{\min}(E^2+R^2)}\leq d_{\hi}^{\mathcal{L}}\big(\lambda E - J+R\big)=d_{\inst}^{\mathcal{L}}\big(\lambda E - J+R\big)
\leq  \sqrt{2\cdot\lambda_{\min}(E^2+R^2)}.\label{dist_instab}
\end{equation}
\end{enumerate}
\end{theorem}
With formulas and close upper and lower bounds available, these distances
can be computed  by global constrained optimization methods such as \cite{hanso}. 
Based on our experience in computing the distance to instability for general matrix pencils, where different computational methods  were studied and it was shown that gradient flow methods were extremely efficient, in the next section we introduce such gradient methods to compute the discussed structured distances.

\section{ODE-based gradient flow approaches} \label{sec:optB}
In the previous section we have seen that for dH pencils the distance to singularity is characterized
by the distance to the nearest common nullspace of three structured matrices and the distance to high index and instability by the distance to the nearest common nullspace of two symmetric positive definite matrices, with perturbations that keep the structure.

The perturbation matrices that give the structured distance to singularity  can be alternatively expressed as
\begin{eqnarray}
(\Delta E_*,\Delta R_*,\Delta J_*) &= &
\arg\min\limits_{\Delta E,\Delta R,\Delta J} \| (\Delta E, \Delta R, \Delta J) \| \nonumber \\
& \text{subj. to } & E+\Delta E , R+\Delta R \in \Sym^{n,n}_{\geq0}, \Delta J \in \Skew^{n,n},  \nonumber \\
& \text{and  }  & (E + \Delta E) x = 0,  \quad  (R + \Delta R) x = 0, \quad  (J + \Delta J) x = 0 \nonumber \\
&& \mbox{for some} \ x \in \R^{n}, x \neq 0.
\label{eq:problemB}
\end{eqnarray}
Then $d_{\sing}^{\mathcal{L}}\big(\lambda E - J+R\big) = \| (\Delta E_*,\Delta R_*, \Delta J_*) \|$ and our algorithmic approach to minimize this functional is based on this reformulation.

\subsection{A two-level minimization}\label{sec:2-level}
To  determine the minimum in \eqref{eq:problemB} we use a two-level minimization. As an  inner iteration, for a perturbation size  $\eps$, we consider perturbed matrices $E + \eps \Delta$, $R + \eps \Theta$ and $J + \eps \Gamma$ with $\| (\Delta, \Theta, \Gamma) \|_F \le 1$ satisfying the constraints in \eqref{eq:problemB}.
Let us denote
\begin{itemize}
\item[ (i) ] by $(\lambda,x)$ an eigenvalue/eigenvector pair of $E + \eps \Delta$ associated with the smallest eigenvalue and $\|x\|=1$;
\item[(ii) ] by $(\nu,u)$ an eigenvalue/eigenvector pair of $R + \eps \Theta$ associated with the smallest eigenvalue, and $\|u\|=1$;
\item[(iii-a) ] if $n$ is even, by $(\iu \mu,w)$ an eigenvalue/eigenvector pair of $J + \eps \Gamma$, with $\mu > 0$ such that $\iu \mu$ is the  eigenvalue with smallest positive imaginary part and $\|w\|=1$,
\item[(iii-b) ]  if $n$ is odd, by $(0,w)$  an eigenvalue/eigenvector pair of $J + \eps \Gamma$ (this exists for all $\Gamma$).
\end{itemize}

In the inner iteration, for any fixed $\eps$ we compute a (local) minimizer of \eqref{eq:problemB} that is, however, different for even or odd $n$.

\subsection*{The case that $n$ is odd} \label{sec:nodd}
In this case the skew-symmetric matrix always has a zero eigenvalue (with an associated real eigenvector $w$) so that the only contribution to the optimization  is through the alignment of $w$ with $x$ and $u$.
Hence,  the functional to be  minimized in (\ref{eq:problemB}) can be expressed in  the simplified form
%
\begin{equation} \label{Fodd}
F_\eps^\od(\Delta,\Theta,\Gamma) = \frac12 \Big(\lambda^2 + \nu^2 +
1-( x^\top u )^2 + 1-( x^\top w )^2 \Big) 
\end{equation}
with $\| \left( \Delta, \Gamma, \Theta \right) \|_F \le 1$.
It is, however,  possible to include a further term $1-| u^\top w |^2$ in the functional, which does not change
the solution but may have an  impact on the conditioning of the problem and hence the numerical performance.

\subsection*{The case that $n$ is even} \label{sec:neven}
{
In this case, when two eigenvalues $\pm \iu \mu$ ($\mu>0$) coalesce at $0$, they form
a semi-simple double eigenvalue and the associated eigenvectors $w=w_1+\iu w_2$ and $\conj{w}=w_1-\iu w_2$ form a two-dimensional
nullspace spanned by the two real vectors $w_1$ and $w_2$. These can be assumed to be orthogonal to each other, i.e. $w_1^\top w_2=0$ and have the same norm $1/\sqrt{2}$ so that still $\|w\|=1$. Using $w_1,w_2$, we define the real orthogonal matrix
\begin{equation*}
W = \sqrt 2 \left [ w_1, w_2 \right ],
\end{equation*}
and to satisfy the constraint in \eqref{eq:problemB}, we require that
\begin{equation*}
W z = x \qquad \mbox{for some} \ z \in \R^2.
\end{equation*}
This leads to the minimization of
\begin{equation*}
\| W z - x \| \qquad \mbox{for some} \ z \in \R^2.
\end{equation*}
Since $W$ is orthogonal, the solution is $z = W^\top x$, and the functional to be minimized takes the form
\begin{equation*}
1 - x^\top W W^\top x = 1- 2\,(x^\top w_1)^2 - 2\,(x^\top w_2)^2,
\end{equation*}
which is positive if $x$ does not lie in the range of $W$ and zero otherwise.

In summary, the functional in the even case is given by
%
\begin{equation} \label{Feven}
F_\eps^\ev(\Delta,\Theta,\Gamma) = \frac12 \Big(\lambda^2 + \nu^2 + \mu^2 +
1-( x^\top u )^2 + 1 - 2\,(x^\top \Re(w))^2 - 2\,(x^\top \Im(w))^2 \Big) 
\end{equation}
with $\| \left( \Delta, \Gamma, \Theta \right) \|_F \le 1$.
\medskip
\begin{remark}\label{rem:real}{\rm
In both the odd and the even case we have that
\[
\min_{\| \left( \Delta, \Gamma, \Theta \right) \|_F \le 1} F_\eps(\Delta,\Theta,\Gamma) = \min_{\| \left( \Delta, \Gamma, \Theta \right) \|_F = 1} F_\eps(\Delta,\Theta,\Gamma).
\]
To see this, consider  $(\Delta_*,\Theta_*,\Gamma_*)$ of Frobenius norm less than or equal to $1$ giving  a minimizer of the left-hand side and suppose that $(x_*,\lambda_*)$ is the minimizing eigenvalue/eigenvector pair of $E+\eps \Delta_*$. Then choosing a matrix $\widetilde\Delta\ne 0$ such that $\widetilde\Delta x_*=0$ and $\langle  \Delta_*, \widetilde\Delta \rangle =0$, for a suitable $\theta$ the matrix $\Delta_1=\Delta_*+\theta\widetilde\Delta$   is of unit Frobenius norm and has
$F_\eps(\Delta_1, \Theta_*,\Gamma_*)=F_\eps(\Delta_*, \Theta_*,\Gamma_*)$.
}
\end{remark}
\medskip

Using the functionals \eqref{Fodd}, respectively \eqref{Feven},  in our approach the local minimizer of $\min_{\| \left( \Delta, \Gamma, \Theta \right) \|_F = 1} F_\eps(\Delta,\Theta,\Gamma)$ is determined as an equilibrium point of the associated gradient system. Note, however,  that in general  this may not be a global minimizer.

For the outer iteration  we consider a continuous branch, as a function of $\eps$,  of the minimizers
$\left( \Delta(\eps), \Gamma(\eps), \Theta(\eps) \right)$ and vary $\eps$ iteratively
%
in order find the smallest solution of the scalar equation
\[
f(\eps)=F_\eps\left(\left( \Delta(\eps), \Gamma(\eps), \Theta(\eps) \right) \right) = 0
\]
with respect to $\eps$.
%
\begin{figure}[h]
\centerline{
\includegraphics[scale=0.43]{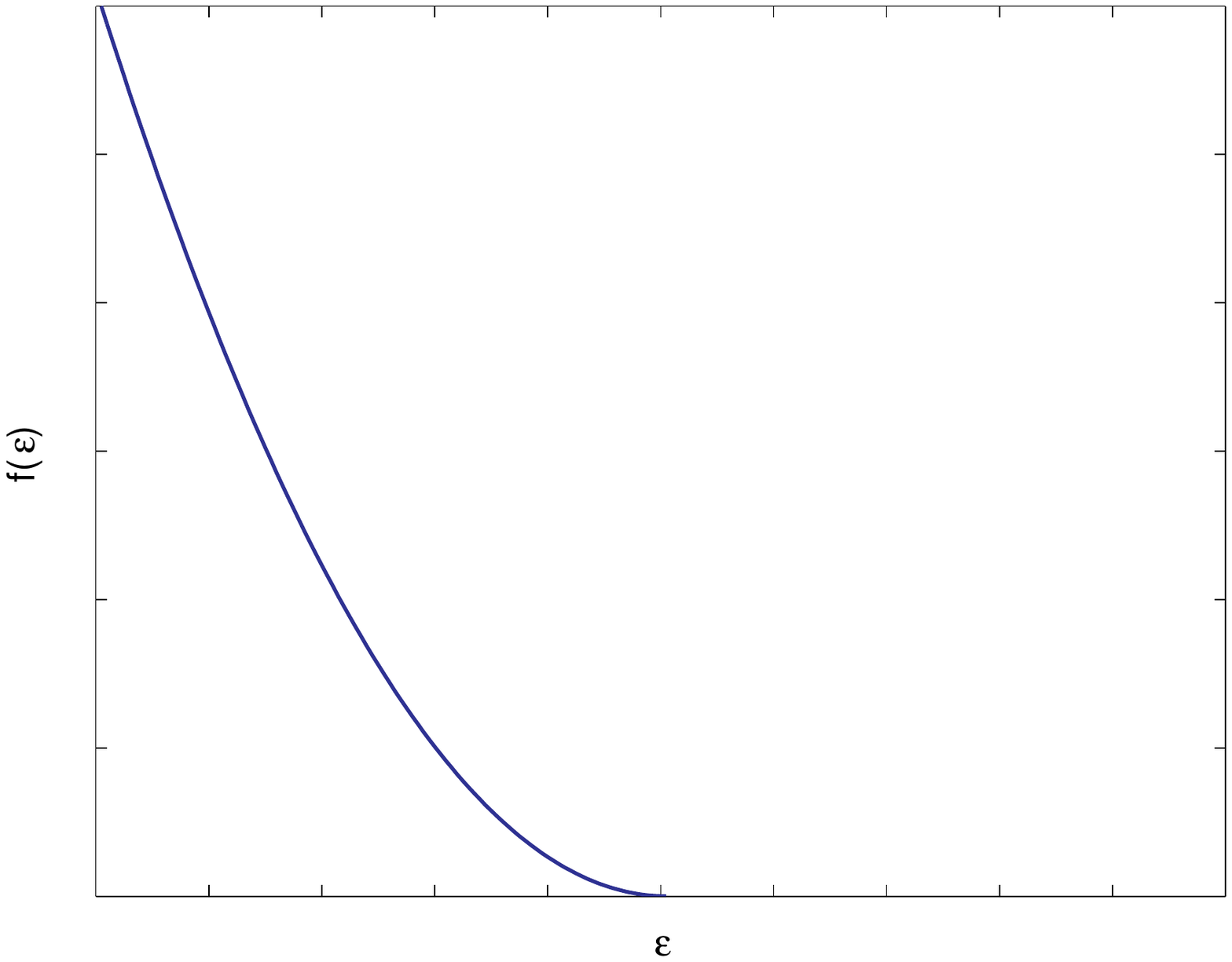}
}
\vspace{-3cm}
\caption{The function $f(\eps)$ in a neighbourhood of $\eps^\star$. For $\eps \ge \eps^\star$
it becomes identically zero. \label{fig_model}
}
\end{figure}

\begin{remark}\label{rem:other cases}{\rm
Note that the techniques for the distance to higher index or instability
follow directly by setting $J=0$ and not perturbing it.
}
\end{remark}

\subsection{Derivatives of eigenvalues and eigenvectors}
The considered minimization is an eigenvalue optimization problem. We will solve this problem by integrating a
differential(-algebraic)  equation  with trajectories that follow the gradient descent and satisfy further constraints. To develop such a method, we first recall a classical  result, see e.g. \cite{Kat95}, for the derivative of a simple eigenvalue
and an associated eigenvector of a matrix $C(t)$ with respect to variations in a real parameter $t$ of the entries.
Here we use the notation $\dot{C}(t):= \frac{d}{dt} C(t)$ to denote the derivative with respect to~$t$.
\begin{lemma} {\rm \cite[Section II.1.1]{Kat95}} \label{lem:eigderiv}
Consider a continuously differentiable matrix valued function
$C(t) :\mathbb \R \to \mathbb \R^{n,n}$, with $C(t)$ normal (i.e., $C(t)C(t)^\top=C(t)^\top C(t)$ for all $t$).
Let $\lambda(t)$ be a simple eigenvalue of $C(t)$ for all $t$ and let $x(t)$ with $\|x(t)\|=1$
be the associated (right and left) eigenvector. Then $\lambda (t)$ is differentiable with
\begin{equation}
\dot\lambda(t) = x(t)^H \dot{C}(t) x(t).
\label{eq:deigval}
\end{equation}
\end{lemma}

For $A\in \Sym^{n,n}$ consider a perturbation matrix $\eps \Delta(t) \in \Sym^{n,n}$ that depends on a real parameter $t$. By Lemma~\ref{lem:eigderiv}, for a simple eigenvalue $\lambda(t) \in \R$ of $A+\eps \Delta(t)$
with associated eigenvector $x(t)$,  $\|x(t)\|=1$, we have (omitting the dependence on $t$)
\begin{equation}
\frac12 \frac d{dt}\l^2 =  \eps\,\l\,x^\top \dot \Delta x = \eps\,\l\,\langle x x^\top, \dot \Delta \rangle.
\label{eq:lambda1}
\end{equation}
%
%
%
Similarly, which is  needed in the case that $n$ is even,  for all $t$, if   $\iu \mu(t) \in \iu \R$ ($\mu(t) \ge 0$) is  a simple eigenvalue of a  matrix-valued function
$B+\eps \Theta(t) \in \Skew^{n,n}$, with associated eigenvector $w(t)$, $\| w(t) \| = 1$, then we have
\begin{equation}
\frac12 \frac d{dt}|\mu|^2 = \eps\,\mu\,\langle \iu w w^H, \dot \Theta \rangle = - \eps\,\mu\,\langle \Im \left(w w^H\right), \dot \Theta \rangle.
\label{eq:lambda2}
\end{equation}
To derive the gradient system associated with our optimization problem, we make use of the following definition.
\begin{definition}
\label{def:groupinv}
Let $M\in {\mathbb C}^{n,n}$ be a singular matrix with a simple zero eigenvalue.
The \emph{group inverse}(reduced resolvent) of $M$ is the unique matrix
$G$ satisfying
\[
MG =GM, \qquad GMG=G, \quad \mbox{and} \quad MGM=M.
\]
\end{definition}
It is well-known, see \cite{MeyS88}, that for a singular and normal matrix  $M\in {\mathbb C}^{n,n}$ with simple eigenvalue zero, its
 group inverse $G$ is equal to the Moore-Penrose pseudoinverse $M^+$. We have the following Lemma.
\begin{lemma} \label{lem:eigvecderiv} {\rm \cite[Theorem 2]{MeyS88}}
Consider  a sufficiently often differentiable  matrix function
\[
C:\mathbb R \to \mathbb C^{n, n}.
\]
Let $\lambda(t)$ be a simple eigenvalue of $C(t)$ for all $t$ and let $x(t)$, with $\|x(t)\|=1$
be the associated right eigenvector function.
Moreover, let $M(t) = C(t) - \lambda (t) I$ and let $G(t)$ be the group inverse of $M(t)$. Then
$x(t)$ satisfies the system of differential equations
\begin{equation}
\dot x  =  x\,x^H G(t)\dot M(t) x  - G(t) \dot M(t) x.
\label{eq:deigvecg}
\end{equation}
Moreover, if $C(t)$ is pointwise normal, then
\begin{equation}
\dot x(t) = - G(t) \dot M(t) x(t)
\label{eq:deigvec}
\end{equation}
\end{lemma}
After these preparations, in the following sections we  determine the associated gradient systems for the functionals \eqref{Fodd}, respectively \eqref{Feven}.

\section{Gradient flow, odd state dimension}\label{sec:gradnodd}
In this section we consider the case that the state dimension  is odd and construct the gradient system optimization algorithm for the functional \eqref{Fodd}.

\subsection{Computation of the gradient}\label{sec:compgrad}
The functional $F_\eps^\od(\Delta,\Theta,\Gamma)$ in \eqref{Fodd} has several  parts.
Applying Lemma~\ref{lem:eigderiv} for perturbations $\eps \Delta(t)$ of $E$ and $\eps \Theta (t)$ of $R$,  the computation of the gradient of the part $\frac12 \Big(\lambda^2 + \nu^2\Big )$ is obtained from the expressions
\begin{eqnarray*}
\frac12 \frac d{dt} \lambda^2  =  \eps\,\l\, \langle x x^\top, \dot \Delta \rangle,\quad
\frac12 \frac d{dt} \nu^2  =  \eps\,\nu\, \langle u u^\top, \dot \Theta \rangle.
\end{eqnarray*}
Considering orthogonal projections with respect to the Frobenius inner product onto the  matrix manifold
$\Sym^{n,n}$,
we identify the constrained gradient directions of these terms as
\begin{equation*}
\dot \Delta \propto \l\,   x x^\top, \qquad
\dot \Theta \propto \nu\,  u u^\top,
\end{equation*}
respectively. (Here $\propto$ denotes proportionality.) In order to treat the other terms, we observe that
\begin{equation*}
\frac12 \frac d{dt} \left(| x^\top u |^2 \right) =
\frac12 \frac d{dt} \left( x^\top u u^\top x \right) =
x^\top u u^\top \dot x  + u^\top x x^\top \dot u,
\end{equation*}
and thus
\begin{eqnarray*}
\frac12 \frac d{dt} \left(1-| x^\top u |^2 \right)
& = & \eps \left(
(x^\top u) u^\top G \dot \Delta x  +
(u^\top x) x^\top N \dot \Theta u
\right)
\nonumber
\\
& = & \eps \left(
\Big\langle 
\theta\, G^\top u x^\top, \dot \Delta \Big\rangle +
\Big\langle 
\theta\, N^\top x u^\top, \dot \Theta \Big\rangle
\right),
\label{eq:xtu}
\end{eqnarray*}
where $\theta = x^\top u$, $G$ is the pseudoinverse of $E+\eps \Delta - \lambda I$, and $N$ is the pseudoinverse of
$R+\eps \Theta - \nu I$.

Since $n$ is odd, which means that (generically) $0$ is a simple eigenvalue of $J + \eps \Gamma$,
for the last term of  \eqref{Fodd} we have
\begin{eqnarray*}
\frac12 \frac d{dt} \left(1-( x^\top w )^2 \right)
& = &
\eps \left(
\Big\langle 
\eta\, G^\top w x^\top, \dot \Delta \Big\rangle +
\Big\langle 
\eta\, P^\top x w^\top, \dot \Gamma \Big\rangle
\right),
\end{eqnarray*}
where $\eta = x^\top w$ and $P$ is the pseudoinverse of $J+\eps \Gamma$.

\subsection{The gradient system of ODEs for the flow in the odd case}\label{sec:odd}
In order to compute the steepest descent direction, we minimize the gradient of $F_\eps$ and collect the
summands involving $\dot \Delta$, $\dot \Theta$ and those involving $\dot \Gamma$.
Letting
\begin{eqnarray}
\nonumber
p & = & 
\theta G^\top u \ + 
\eta G^\top w,
\\
\label{eq:pqr}
q & = & 
\theta N^\top x,
\\
r & = & 
\eta P^\top x,
\nonumber
\end{eqnarray}
we have
\begin{eqnarray}
\frac d{dt} F_\eps(\Delta,\Theta,\Gamma) & = &
\eps\,\langle \left(\lambda x + p \right)\,x^\top, \dot \Delta \rangle
+
\eps\,\langle \left( \nu u + q \right)\,u^\top, \dot \Theta \rangle
+
\eps\,\langle  r \,w^\top, \dot \Gamma \rangle
\label{eq:freegrad}
\\
& = &
\eps \left( \Sym  \left(\langle \left(\lambda x + p \right)\,x^\top \right), \dot \Delta \rangle +
\langle \Sym  \left( \left( \nu u + q \right)\,u^\top \right), \dot \Theta \rangle +
\eps\,\langle  \Skew  \left( r \,w^\top \right), \dot \Gamma \rangle \right),
\nonumber
\end{eqnarray}
where we have used the structural properties of $\dot \Delta, \dot \Theta$ (symmetric) and $\dot \Gamma$
(skew-symmetric) and the property that for real matrices $A$ and $B$, $\langle \Sym(A), \Skew(B) \rangle = 0$.
Equation \eqref{eq:freegrad} identifies the 
gradient of the functional,
\begin{equation} \label{eq:G}
\G = \left( \Sym\left( \left(\lambda x + p \right)\,x^\top\right),
            \Sym\left( \left( \nu u + q \right)\,u^\top \right),
						\Skew\left( r \,w^\top \right) \right) := \left( \G_E, \G_R, \G_J \right).
\end{equation}
Since we want to impose a norm constraint on the perturbation $(\Delta,\Theta,\Gamma)$ we need the
following result.
\begin{lemma}[Direction of steepest admissible ascent]
\label{lem:opt}
Let $\G =\left( \G_E, \G_R, \G_J \right) \in\R^{n, 3 n}$, $\Z = \left( Z_1,Z_2,Z_3 \right), \M=\left( \Delta, \Theta,\Gamma \right) \in\R^{n, 3 n}$
with ${\|(\Delta,\Theta,\Gamma)\|_F=1}$.
A solution of the optimization problem
\begin{eqnarray}
\Z_\star  & = & \arg\min_{\|\Z\|_F=1,\,\,\langle \M, \Z \rangle=0} \ \langle  \G,  \Z \rangle \\
& \text{subj. to } & \Delta , \Theta \in \Sym^{n,n}, \Gamma \in \Skew^{n,n},  \nonumber \\
\label{eq:opt}
\end{eqnarray}
is given by 
\begin{eqnarray}
\mu \Z_\star & = & -\G + \varrho\, \M,
\label{eq:Zopt}
\\
\varrho & = & \left( \langle \Delta,  \Sym  \left( (\lambda x  + p)\,x^\top \right) \rangle +
                     \langle \Theta,  \Sym  \left( (\nu  u  + q)\,u^\top \right) \rangle +
					           \langle \Gamma,  \Skew \left(  r\,w^\top \right) \rangle \right)
\nonumber
\end{eqnarray}
where $\mu$ is the Frobenius norm of the matrix on the right-hand side. 
The solution is unique if $\G$ is not a multiple of $\M$.
\end{lemma}
\begin{proof}
The result follows on noting that the function to minimize is a real inner product on $\R^{n,3n}$, and the real inner product with a given vector
(which here is a matrix) is minimized over a subspace by orthogonally projecting the vector onto that subspace. The expression in (\ref{eq:Zopt}) is
the orthogonal projection of $G$ to the tangent space at $\M$ of the manifold of matrices of unit Frobenius norm.
\end{proof}

Taking into consideration projection with respect to the Frobenius inner product of the vector field onto the manifolds of symmetric and skew-symmetric matrices, this leads to the system of differential equations for the perturbation matrices
\begin{eqnarray}
\dot \Delta & = & - \Sym  \left( \left(\lambda x  + p \right)\,x^\top \right)  + \varrho \Delta, \nonumber
\\[2mm]
\dot \Theta & = & - \Sym  \left( \left(\nu  u  + q \right)\,u^\top \right)  + \varrho \Theta,
\label{eq:ode-o}
\\[2mm]
\dot \Gamma & = & - \Skew \left( r \,w^\top \right)  + \varrho \Gamma, \nonumber
\end{eqnarray}
where, for $X \in \R^{n,n}$, $\Sym(X) = \frac{X+X^\top}{2}$, $\Skew(X) = \frac{X-X^\top}{2}$, and
\[
\varrho = \left( \langle \Delta,  \Sym  \left( (\lambda x  + p)\,x^\top \right) \rangle +
                     \langle \Theta,  \Sym  \left( (\nu  u  + q)\,u^\top \right) \rangle +
					           \langle \Gamma,  \Skew \left(  r\,w^\top \right) \rangle \right)
\]
is used to ensure the norm conservation, i.e. $\langle ( \dot\Delta, \dot\Theta, \dot\Gamma ), \left( \Delta, \Theta, \Gamma \right) \rangle = 0$.

\begin{theorem} \label{thm:monotone}
Let $(\Delta(t),\Theta(t),\Gamma(t))$ of unit Frobenius norm satisfy the differential equation \eqref{eq:ode-o}.
If $\l(t)$ is a simple eigenvalue of $A+\eps \Delta(t)$, then
\begin{eqnarray}
\frac{d}{dt} F_\eps\left( \Delta(t),\Theta(t),\Gamma(t) \right) & \le & 0.
\label{eq:neg}
\end{eqnarray}
\end{theorem}
\proof
The result follows directly by the fact that \eqref{eq:ode-o} is a constrained gradient system.
\eproof

In this way we have preserved the symmetry of $E,R$ and the skew-symmetry of $J$.
It may happen however, that along the solution trajectory of \eqref{eq:ode-o}, due to the projection on the matrix manifolds the smallest eigenvalue
$\nu$ of $R + \eps \Theta$ and/or the smallest eigenvalue $\lambda$ of $E+\eps \Delta$ become negative. In this case the perturbed system is not a dissipative Hamiltonian system any longer. This, however, is in general not an issue for the optimization algorithm, since the dynamical gradient system leads to eigenvalues
$\nu$, $\lambda$, with $|\nu|$ as small as possible, for a given $\eps$, and thus drives them to zero when $\eps=\eps^*$, so that in the limiting situation also the positive semidefiniteness of $E + \eps \Delta$ and $R+\eps \Theta$ holds.

\subsection{Stationary points of \eqref{eq:ode-o} and low rank property}\label{sec:statpt}
In this subsection we discuss the existence of stationary points of  the solution trajectory of \eqref{eq:ode-o}.
%
%
\begin{lemma}
\label{lem:nonzero} Let $\eps$ be fixed and $F_\eps^\od(\Delta,\Theta,\Gamma) > 0$. Let $\lambda$ be a simple eigenvalue  of $E+\eps \Delta$ with associate normalized eigenvector $x$, let $\nu$ be a simple eigenvalue  of $R+\eps \Theta$ with  associate normalized eigenvector $u$,  and let $0$ be a simple eigenvalue  of $J+\eps \Gamma$ with associate normalized eigenvector $w$. Then, in the generic situation, i.e., if
 $\lambda,\nu \neq 0$, $\theta = x^\top u \neq 0$, and $\eta = x^\top w \in (0,1)$, we have
\begin{equation}
\lambda x + p \neq 0, 
\qquad  \nu u  + q \neq 0, \qquad \mbox{and} \qquad  r \neq 0.
\label{eq:ayvbbw}
\end{equation}
\end{lemma}
\begin{proof}
The proofs for the three cases are similar.
\begin{itemize}
\item[(i) ] 
Exploiting the property that $G x = 0$, see \cite{MeyS88}, we obtain that
$x^\top p = 0$.
If we had
$x^\top \left( \lambda x + p \right) = 0$, then this would imply  $\lambda x^\top x = 0$ and thus, since $\lambda \neq 0$, we get a contradiction, since $\| x\|=1$.
\item[(ii) ] Exploiting the property $N u = 0$, we obtain that
$u^\top q = 0$.
If we had $u^\top \left( \nu u  + q \right) = 0$, then  we get  $\nu u^\top u =0$, and again we have a contradiction.
\item[(iii) ] Having assumed $\eta = x^\top w \neq 1$ we have that $x$ and $w$ are not aligned. As a consequence
$\eta P^\top x \neq 0$. \eproof
\end{itemize}
Using Lemma~\ref{lem:nonzero}, we have the following characterization of stationary points..
\begin{theorem}\label{stat:odd}
Let $(\Delta(t),\Theta(t),\Gamma(t))$ of unit Frobenius norm satisfy the differential equation \eqref{eq:ode-o}.
Moreover, suppose that for all $t$
\[
F_\eps^\od(\Delta(t),\Theta(t),\Gamma(t)) > 0.
\]
and that  $0 \neq \lambda(t) \in \R$ is a simple eigenvalue of $E+\eps \Delta(t)$ with normalized eigenvector $x(t)$,
that $0 \neq \nu(t) \in \R$ is a simple eigenvalue of $R+\eps \Theta(t)$ with associated eigenvector $u(t)$,
and that $J+\eps \Gamma(t)$ has a null vector $w(t)$.

Then the following are equivalent (here we omit the argument~$t$):
\begin{itemize}
\item[{\rm ($1$)} ] $\displaystyle{\frac{d}{d t}} F_\eps^\od(\Delta,\Theta,\Gamma) =0$;
\smallskip
\item[{\rm ($2$)} ] $\dot \Delta=0$, $\dot \Theta=0$, $\dot \Gamma=0$;
\smallskip
\item[{\rm ($3$)} ] $\Delta$ is a 
multiple of the rank-$2$ matrix $\Sym  \left( \left(\lambda x  + p \right)\,x^\top \right)$;
$\Theta$ is a 
multiple of the rank-$2$ matrix $\Sym  \left( \left(\nu  u  + q \right)\,u^\top \right)$;
$\Gamma$ is a 
multiple of the rank-$2$ matrix $\Skew \left( r \,w^\top \right)$ with $p,q,r$ given by
\eqref{eq:pqr}.
\end{itemize}
\end{theorem}

The proof follows directly by equating to zero the right hand side of \eqref{eq:ode-o} and by Lemma
\ref{lem:nonzero}, which prevents the matrices to be zero.
\end{proof}

We have also the following extremality property.
\begin{theorem}\label{thm:ext}
Consider the functional \eqref{Fodd} and suppose that $F_\eps^\od(\Delta,\Theta,\Gamma) > 0$.
Let $\Delta_* \in \Sym^{n,n},  \Theta_* \in \Sym^{n,n}$ and $\Gamma_* \in \Skew^{n,n}$
with $\| \left(\Delta_*,\Theta_*,\Gamma_*\right) \|_F = 1$.
Let $0 \neq \lambda_* \in \R$ be a simple eigenvalue of $E+\eps \Delta_*$ with associated eigenvector $x$,
let $0 \neq \nu_* \in \R$ be a simple eigenvalue of $R+\eps \Theta_*$ with associated eigenvector $u$, and
let $J+\eps \Gamma_*$ have a null vector $w$.
Then the following  are equivalent:
\begin{itemize}
\item[{\rm (i)} ] Every differentiable path $(\Delta(t),\Theta(t),\Gamma(t))$ (for small $t\ge 0$)
with the properties that
$\| (\Delta(t),\Theta(t),\Gamma(t)) \|_F\le 1$, that both $\lambda(t)$ and $\nu(t)$
are simple eigenvalues of $E+\eps \Delta(t)$ and $R+\eps \Theta(t)$,  with associated eigenvectors $x(t)$ and $u(t)$, respectively, and
for which $w(t)$ is the null vector of $J + \eps \Gamma(t)$,
so that $\Delta(0)=\Delta_{*}, \Theta(0)=\Theta_{*}, \Gamma(0)=\Gamma_{*}$, satisfies
\[
\frac{d}{dt} F_\eps^\od(\Delta(t),\Theta(t),\Gamma(t))  \ge 0.
\]
\item[{\rm (ii)} ]
The matrix $\Delta_{*}$ is a 
multiple of the rank two matrix $\Sym  \left( \left(\lambda x  + p \right)\,x^\top \right)$,
$\Theta_*$ is a 
multiple of the rank two matrix $\Sym  \left( \left(\nu  u  + q \right)\,u^\top \right)$, and
$\Gamma_*$ is a 
multiple of the rank two matrix $\Skew \left( r \,w^\top \right)$ with $p,q,r$ given by
\eqref{eq:pqr}.
\end{itemize}
\end{theorem}
\begin{proof}
Lemma \ref{lem:nonzero} ensures that $\lambda x  + p \neq 0$.

Assume that (i) does not hold. Then there exists a path $(\Delta(t),\Theta(t),\Gamma(t))$ through $(\Delta_*,\Theta_*,\Gamma_*)$
such that $\frac{d}{dt} F_\eps(\Delta(t), \Theta(t), \Gamma(t))\big|_{t=0} < 0$.
The steepest descent gradient property shows that also the solution path of (\ref{eq:ode-o}) passing through $(\Delta_*,\Theta_*,\Gamma_*)$
is such a path.

Hence $(\Delta_*,\Theta_*,\Gamma_*)$ is not a stationary point of (\ref{eq:ode-o}), and Theorem~\ref{stat:odd} then yields that
(ii) does not hold.

Conversely, if
\begin{equation*}
\left(\Delta_*,\Theta_*,\Gamma_*\right) \not\propto
\left( \Sym  \left( \left(\lambda x  + p \right)\,x^\top \right),\Sym  \left( \left(\nu  u  + q \right)\,u^\top \right)\Skew \left( r \,w^\top \right) \right)
\end{equation*}
then $(\Delta_*,\Theta_*,\Gamma_*)$ is not a stationary point of (\ref{eq:ode-o}), and Theorems~\ref{stat:odd} and \ref{thm:monotone} yield that
$\frac{d}{dt} F_\eps(\Delta(t), \Theta(t), \Gamma(t))\big|_{t=0} < 0$ along the solution path of (\ref{eq:ode-o}).
\end{proof}

\subsection{Sparsity preservation}\label{sec:sparsity}

If the matrices $E,R$ and $J$ have a given sparsity pattern, then we may include as a constraint that the perturbations do not alter the sparsity structure.
In terms of the Frobenius norm, it is immediate to obtain the constrained gradient system. Denoting by $\Pi_E$, $\Pi_R$, and $\Pi_J$, respectively,  projections
onto the manifold of sparse  matrices with the given sparsity pattern and structure of $E$, $R$ and $J$, then we get
\begin{eqnarray}
\dot \Delta & = & - \Pi_E \Sym \left( \left(\lambda x  + p \right)\,x^\top \right)  + \varrho \Delta, \nonumber
\\[2mm]
\dot \Theta & = & - \Pi_R \Sym \left( \left(\nu  u  + q \right)\,u^\top \right)  + \varrho \Theta,
\label{eq:ode-osp}
\\[2mm]
\dot \Gamma & = & - \Pi_J \Skew \left( r \,w^\top \right)  + \varrho \Gamma, \nonumber
\end{eqnarray}
where
\[
\varrho =     \left( \langle \Delta,  \Pi_E \Sym  \left( (\lambda x  + p)\,x^\top \right) \rangle +
                     \langle \Theta,  \Pi_R \Sym \left( (\nu  u  + q)\,u^\top \right) \rangle +
					           \langle \Gamma,  \Pi_J \Skew \left(  r\,w^\top \right) \rangle \right).
\]

After deriving formulas, in the next subsection we illustrate the properties of the optimization procedure with a numerical example.
\subsection{A numerical example} \label{sec:illex}

Let $n=5$ and consider the randomly generated matrices
\begin{eqnarray*}
E   &=& \left[ \begin{array}{rrrrr}
    0.15  &  0.02  & -0.04  &  0.02  & -0.04 \\
    0.02  &  0.22  &     0  & -0.01  & -0.03 \\
   -0.04  &     0  &  0.11  & -0.07  & -0.04 \\
    0.02  & -0.01  & -0.07  &  0.01  &  0.10 \\
   -0.04  & -0.03  & -0.04  &  0.10  &  0.39
\end{array} \right], \\
R   &=& \left[ \begin{array}{rrrrr}
    0.49 & -0.13 &  0.05 & -0.15 &  0.11 \\
   -0.13 &  0.23 & -0.05 & -0.10 & -0.19 \\
    0.05 & -0.05 &  0.48 & -0.06 &  0.02 \\
   -0.15 & -0.10 & -0.06 &  0.55 &  0.16 \\
    0.11 & -0.19 &  0.02 &  0.16 &  0.48
\end{array} \right],\\
J   &=& \left[ \begin{array}{rrrrr}
       0  &  -0.27  &  -0.03  &  -0.01  &   0.21 \\
    0.27  &      0  &  -0.15  &   0.03  &   0.11 \\
    0.03  &   0.15  &      0  &   0.07  &  -0.07 \\
    0.01  &  -0.03  &  -0.07  &      0  &   0.05 \\
   -0.21  &  -0.11  &   0.07  &  -0.05  &      0
\end{array} \right].
\end{eqnarray*}
Running the two level iteration with an initial value of the functional
$F_0^\od \left( \cdot,\cdot,\cdot \right) = 0.9181$, we find a perturbation at a distance (rounded to four digits) $\eps^* = 0.35681$ with  a common null space given by the vector
\begin{equation*}
c = \left[ \begin{array}{rrrrr}
    0.2195 \\
   -0.6664 \\
   -0.0639 \\
    0.3187 \\
   -0.6341 \end{array} \right]
\end{equation*}
and the computed perturbations are given by
\begin{eqnarray*}
\Delta E &=& \left[ \begin{array}{rrrrr}
   -0.0385 &  0.0912 &  0.0089 &  0.0251 &  0.1408 \\
    0.0912 & -0.2114 & -0.0218 & -0.0903 & -0.3482 \\
    0.0089 & -0.0218 & -0.0009 &  0.0019 & -0.0293 \\
    0.0251 & -0.0903 &  0.0019 &  0.1628 & -0.0123 \\
    0.1408 & -0.3482 & -0.0293 & -0.0123 & -0.4801
\end{array} \right],\\
\Delta R &=& \left[ \begin{array}{rrrrr}
   -0.0689 &  0.1166 &  0.0148 & -0.1044 &  0.1242 \\
    0.1166 & -0.1164 & -0.0307 &  0.1766 & -0.1433 \\
    0.0148 & -0.0307 & -0.0049 &  0.0208 & -0.0319 \\
   -0.1044  & 0.1766 &  0.0208 & -0.1592 &  0.1884 \\
    0.1242 & -0.1433 & -0.0319 &  0.1884 & -0.1679
\end{array} \right],\\
\Delta J &=& \left[ \begin{array}{rrrrr}
         0 &  0.0887 & -0.0118 & -0.0474 &  0.0852 \\
   -0.0887 &       0 &  0.0496 &  0.0003 &  0.0027 \\
    0.0118 & -0.0496 &       0 &  0.0257 & -0.0488 \\
    0.0474 & -0.0003 & -0.0257 &       0 & -0.0030 \\
   -0.0852 & -0.0027 &  0.0488 &  0.0030 &       0	
\end{array} \right ].
\end{eqnarray*}

\section{Rank two dynamics}\label{sec:rank2}
Theorem~\ref{stat:odd} motivates to search for a differential equation on the manifold of rank two symmetric/skew-symmetric matrices,
which still leads to a gradient system for $F_\eps^\od$, but in addition requires the derivatives of the matrices $\Delta,\Theta,\Gamma$ lying
in the respective tangent spaces.

Let $\cM_2^{n,n} := \{ X\in \R^{n\times n}: {\rm rank}(X)=2 \}$. Then we restrict the perturbations to the matrix manifolds
\begin{equation}
\Delta, \Theta \in \cM_2^{\Sym^{n,n}}, \qquad \Gamma \in \cM_2^{\Skew},
\end{equation}
where $\cM_2^{\Sym^{n,n}} = \cM_2^{n,n} \cap \Sym^{n,n}$ and $\cM_2^{\Skew^{n,n}} = \cM_2^{n,n} \cap \Skew^{n,n}$.

Following \cite{KL08}, every real symmetric rank two matrix $X$ of dimension $n\times n$ can be written in
the form
\begin{equation}\label{USU}
X = USU^\tp,
\end{equation}
where $U\in\R^{n\times 2}$ has orthonormal
columns, i.e., $ U^\tp  U = I_2$
and $S\in \Sym^{2\times 2}$. Here we will not assume that $S$ is diagonal.
Note that the representation (\ref{USU}) is not unique; indeed replacing $U$ by
$\widetilde U=U U_1$ with orthogonal $U_1\in\R^{2\times 2}$
and correspondingly $S$ by $\widetilde S=U_1^\tp S U_1$, yields the same matrix
$X=USU^\tp =\widetilde U \widetilde S \widetilde U^\tp $.

As a compensation for the non-uniqueness in the decomposition
(\ref{USU}), we will use a unique decomposition in the tangent space.
Let $\V_{n,2}$ denote the Stiefel manifold of real $n\times 2$ matrices
with orthonormal columns. The tangent space at $U\in\V_{n,2}$ is given by
\begin{eqnarray*}
T_U \V_{n,2} &=&
\{ \dot  U \in \R^{n\times 2}:\
\dot  U^\tp  U + U^\tp  \dot  U = 0\}
=
\{ \dot  U \in \R^{n\times 2}:\
U^\tp  \dot  U \, \hbox{is skew-symmetric}\}.
\end{eqnarray*}
Following  \cite{KL08}, every tangent matrix $\dot  X \in T_X\cM_2^{\Sym^{n,n}}$ is of the form
\[ 
\dot  X = \dot U S U^\tp  + U\dot  S U^\tp  + U S \dot U^\tp ,
\]
where  $\dot  S\in\Sym^{2\times 2}$,  $\dot  U \in T_U \V_{n,2}$,
and $\dot  S, \dot  U$ are uniquely determined by $\dot  X$ and $U,S$, if we impose the orthogonality condition
\[
U^\tp \dot  U = 0.
\]
We note the following lemma adapted from \cite{KL08}.
\begin{lemma}\label{lem:P-formula}
The orthogonal projection onto the tangent space $T_X\cM_2^{\Sym^{n,n}}$ at $X=USU^\tp \in \cM_2^{\Sym^{n,n}}$
is given by
\begin{equation}\label{P-formula}
P^\Sym_X(Z) = Z - (I-UU^\tp) Z (I-UU^\tp)
\end{equation}
for $Z\in\Sym^{n\times n}$.
\end{lemma}
Analogous results hold for $Y \in \Skew^{n,n}$, with $S \in \Skew^{2,2}$.

\subsection{A differential equation for rank two matrices} \label{rank2ode}
To derive the differential equation in the rank two case, we replace
in (\ref{eq:ode-o})  the right-hand sides by the orthogonal projections to
$T_\Delta\cM_2^{\Sym^{n,n}}$, $T_\Theta\cM_2^{\Sym^{n,n}}$, and $T_\Gamma\cM_2^{\Skew^{n,n}}$, respectively, so that solutions starting with rank-two will retain rank two for all $t$. This gives the differential equations
\begin{eqnarray}
\dot \Delta & = & - P^\Sym_\Delta \Bigl( \Sym  \left( \left(\lambda x  + p \right)\,x^\top \right)  + \varrho \Delta \Bigr), \nonumber
\\[2mm]
\dot \Theta & = & - P^\Sym_\Theta \Bigl( \Sym  \left( \left(\nu  u  + q \right)\,u^\top \right)  + \varrho \Theta \Bigr),
\label{eq:ode-or2}
\\[2mm]
\dot \Gamma & = & - P^\Skew_\Gamma \Bigl( \Skew \left( r \,w^\top \right)  + \varrho \Gamma \Bigr), \nonumber
\end{eqnarray}
where again $p,q$ and $r$ are defined by \eqref{eq:pqr} and
\begin{eqnarray*}
\varrho & = &           \Big\langle \Delta,  P^\Sym_\Delta  \Bigl( \Sym  \left( \left(\lambda x  + p \right)\,x^\top \right) \Bigr) \Big\rangle +
\\
& + &
												\Big\langle \Theta,  P^\Sym_\Theta  \Bigl( \Sym  \left( \left(\nu u  + q \right)\,u^\top \right) \Bigr) \Big\rangle +
					              \Big\langle \Gamma,  P^\Skew_\Gamma \Bigl( \Skew\left( r\,w^\top \right) \Bigr) \Big\rangle  .
\end{eqnarray*}

Since for $X\in \cM_2^{\Sym^{n,n}}$ and $Z \in \Sym^{n,n}$, we have $P_X(X)=X$ and $\langle X,Z \rangle = \langle X,P_X(Z) \rangle$,
(and analogous  properties hold for $X\in \cM_2^{\Skew^{n,n}}$ and $Y \in \cM_2^{\Skew^{n,n}}$), the system of differential equations can be rewritten as
\begin{eqnarray}
\dot \Delta & = & - P^\Sym_\Delta \Bigl( \Sym  \left( \left(\lambda x  + p \right)\,x^\top \right) \Bigr) + \varrho \Delta,  \nonumber
\\[2mm]
\dot \Theta & = & - P^\Sym_\Theta \Bigl( \Sym  \left( \left(\nu  u  + q \right)\,u^\top \right) \Bigr)  + \varrho \Theta,
\label{eq:ode-o2}
\\[2mm]
\dot \Gamma & = & - P^\Skew_\Gamma  \Bigl( \Skew \left( r \,w^\top \right) \Bigr) + \varrho \Gamma,  \nonumber
\end{eqnarray}
with
\begin{eqnarray*}
\varrho & = & \left(     \left\langle \Delta,  P^\Sym_\Delta \left( \Sym  \left( (\lambda x  + p)\,x^\top \right) \right) \right\rangle +
                         \left\langle \Theta,  P^\Sym_\Theta \left( \Sym  \left( (\nu  u  + q)\,u^\top \right) \right) \right\rangle \ \right.
\\										
	      & + & \left. \left\langle \Gamma,  P^\Skew_\Gamma \left( \Skew \left(  r\,w^\top \right) \right) \right\rangle \right).
\end{eqnarray*}
This system differs from \eqref{eq:ode-o} in that  the free gradient terms are replaced by their orthogonal projections on the rank two manifold of the corresponding structure.

To obtain the differential equation in a form that uses the factors in $X=USU^\tp $ rather than the full $n\times n$ matrix $X$, we use the following result,
whose proof is similar to that given in \cite[Prop. 2.1]{KL08}.
\begin{lemma}\label{lem:us}
For $X=USU^\tp \in \cM_2^{\Sym^{n,n}}$ with nonsingular $S\in\Sym^{2\times 2}$ and with
$U\in\R^{n\times 2}$ having orthonormal columns,
the equation $\dot X=P^\Sym_X(Z)$ with $Z$ symmetric is equivalent to
$
\dot X = \dot U S U^\tp  + U \dot S U^\tp  + U S\dot U^\tp ,
$
where
\begin{eqnarray}
\dot S &=& U^\tp  Z U,
\nonumber\\
\dot U &=& (I-UU^\tp) Z U S^{-1}.
\label{odes}
\nonumber
\end{eqnarray}
An analogous statement holds for $Y \in \cM_2^{\Skew^{n,n}}$ and $Z$ skew-symmetric.
\end{lemma}

With the ansatz $Z=  {}-\lambda x x^\top - \frac12 p x^\top + \frac12 x p^\top$, and introducing for $S_1,U_1$, the quantities $g_1 = U_1^\top x \in \R^2$ and $h_1 = U_1^\top p \in \R^2$, this yields that the differential equation (\ref{eq:ode-o2})
for $\Delta=U_1 S_1 U_1^T$ is equivalent to the following system of differential equations
\begin{eqnarray}
\dot S_1 &=&  {}-\lambda g_1 g_1^\top - \frac12 \left( g_1 h_1^\top + h_1 g_1^\top \right)   + \varrho S_1,
\nonumber
\\
\label{ode-su1}
\\
\dot U_1 &=& \left({}-\lambda x g_1^\top - \frac12 \left( x h_1^\top + p g_1^\top \right)
          +  U_1\,\left( \lambda g_1 g_1^\top + \frac12 \left( g_1 h_1^\top + h_1 g_1^\top \right) \right) \right) S_1^{-1}.
\nonumber
\end{eqnarray}
Similarly, for $\Theta=U_2 S_2 U_2^T$, setting  $g_2 = U_2^\top u \in \R^2, h_2 = U_2^\top q \in \R^2$, we obtain the system of differential equations
\begin{eqnarray}
\dot S_2 &=&  {}-\nu g_2 g_2^\top - \frac12 \left( g_2 h_2^\top + h_2 g_2^\top \right)   + \varrho S_2,
\nonumber
\\
\label{ode-su2}
\\
\dot U_2 &=& \left({}-\nu u g_2^\top - \frac12 \left( u h_2^\top + q g_2^\top \right)
          +  U_2\,\left( \lambda g_2 g_2^\top + \frac12 \left( g_2 h_2^\top + h_2 g_2^\top \right) \right) \right) S_2^{-1}.
\nonumber
\end{eqnarray}
Finally, for $\Gamma=U_3 S_3 U_3^T$ (with $S_3 \in \Skew^{2,2}$), setting $g_3 = U_3^\top w \in \R^2, h_3 = U_3^\top r \in \R^2$, we obtain the  system of differential equations
\begin{eqnarray}
\dot S_3 &=&  \frac12 \left( {}-g_3 h_3^\top + h_3 g_3^\top \right)   + \varrho S_3,
\nonumber
\\
&&
\label{ode-su3}
\\
\dot U_3 &=& \left(\frac12 \left( {}-w h_3^\top + r g_3^\top \right)
          +  U_3\,\left( \frac12 \left( g_3 h_3^\top - h_3 g_3^\top \right) \right) \right) S_3^{-1}.
\nonumber
\end{eqnarray}
Having established differential equations for rank two factors, in the next section we discuss the monotonicity of the functional.

\subsection{Monotonicity of the functional}\label{sec:monotone}
We have the following monotonicity result, which establishes that \eqref{eq:ode-o2} is a gradient system for
$F_\eps^\od(\Delta(t),\Theta(t),\Gamma(t))$.

\begin{theorem} \label{thm:monotone-2}
Let $\Delta(t),\Theta(t) \in \cM_2^{\Sym^{n,n}},\Gamma(t) \in \cM_2^{\Skew^{n,n}}$ satisfy the differential equation \eqref{eq:ode-o2}, and suppose that
\begin{equation*}
F_\eps^\od(\Delta(t),\Theta(t),\Gamma(t)) > 0.
\end{equation*}
If $\lambda(t)$ is a simple eigenvalue of $E+\eps \Delta(t)$, $\mu(t)$ is a simple eigenvalue of $R + \eps \Theta(t)$, and $0$
a simple eigenvalue of $J + \eps \Gamma$, then
\begin{equation}
\frac{d}{dt} F_\eps^\od(\Delta(t),\Theta(t),\Gamma(t))  < 0.
\label{eq:mon-2}
\end{equation}
\end{theorem}
\begin{proof}
We note  that
\begin{eqnarray*}
 \frac{1}{2 \eps} \frac d{dt} F_\eps^\od(\Delta,\Theta,\Gamma) &=&
\l\,\langle x x^\top, \dot \Delta \rangle +
\nu\, \langle u u^\top, \dot \Theta \rangle +
\left(
\Big\langle 
p x^\top, \dot \Delta \Big\rangle +
\Big\langle
q u^\top, \dot \Theta \Big\rangle +
\Big\langle 
r w^\top, \dot \Gamma \Big\rangle
\right)
\\[2mm]
& = &
{}-\Big\langle \left(\l x  + p \right)\,x^\top, - P^\Sym_\Delta \Bigl( \Sym  \left( \left(\lambda x  + p \right)\,x^\top \right) \Bigr) \Big\rangle +
\varrho \langle \left(\lambda x  + p \right)\,x^\top, \Delta \rangle
\\
&&
{}-\Big\langle \left(\nu u  + q \right)\,u^\top, - P^\Sym_\Theta \Bigl( \Sym  \left( \left(\nu u  + q \right)\,u^\top \right) \Bigr) \Big\rangle +
\varrho \langle \left(\nu u  + q \right)\,u^\top, \Theta \rangle
\\
&&
{}-\Big\langle r\,w^\top, - P^\Skew_\Gamma \Bigl( \Skew\left( r\,w^\top \right) \Bigr) \Big\rangle +
\varrho \langle r\,w^\top, \Theta \rangle
\\[2mm]
& = &
\varrho^2-\Big\| P^\Sym_\Delta \Bigl( \Sym  \left( \left(\lambda x  + p \right)\,x^\top \right) \Bigr) \Big\|^2 -
   \Big\| P^\Sym_\Theta \Bigl( \Sym  \left( \left(\nu u  + q \right)\,u^\top \right) \Bigr) \Big\|^2 \\
& - &   \Big\| P^\Skew_\Gamma \Bigl( \Skew\left( r\,w^\top \right) \Bigr) \Big\|^2,
\end{eqnarray*}
with
\begin{eqnarray*}
\varrho & = &           \Big\langle \Delta,  P^\Sym_\Delta  \Bigl( \Sym  \left( \left(\lambda x  + p \right)\,x^\top \right) \Bigr) \Big\rangle +
                     \Big\langle \Theta,  P^\Sym_\Theta  \Bigl( \Sym  \left( \left(\nu u  + q \right)\,u^\top \right) \Bigr) \Big\rangle
\\
& + &					           \Big\langle \Gamma,  P^\Skew_\Gamma \Bigl( \Skew\left( r\,w^\top \right) \Bigr) \Big\rangle  .
\end{eqnarray*}
Applying the Cauchy-Schwarz inequality proves the assertion.
\end{proof}

\subsection{Computational approach} \label{sec:comp}

We will use the explicit Euler method to carry out the numerical integration of the gradient systems associated with the rank two perturbations. For this we use  Algorithm~\ref{alg1},  an adaptation of the method proposed  in \cite{CL20}, and solve instead the system given by
%
%
%
$\dot{K}(t) = F \left( t, K(t)\,U_0^\top \right)\,U_0$ on the interval $[t_0,t_1]$ and to approximate the solution at time $t_1$ of the ODE $\dot{S}(t) = U_1^\top\, F \left( t, U_1 S(t) U_1^\top \right)$.
If $n$ is large, then the memory requirement and the computing time are significantly reduced with respect to the integration
of the full ODEs.
\IncMargin{1em}
\begin{algorithm}\label{alg1}
\DontPrintSemicolon
\KwData{Matrix $X_0 = U_0 S_0 U_0^\top$, $F(t,X)$, $t_0,t_1$, $h$}
\KwResult{Matrix $X_1 = U_1 S_1 U_1^\top$}
\Begin{
\nl Solve the $n \times 2$ ODE $\dot{K}(t) = F \left( t, K(t)\,U_0^\top \right)\,U_0, \ K(t_0) = U_0 S_0$.\;
\vskip 3mm
\nl Compute a $Q R$-decomposition  $K(t_1) = U_1\,R$.\;
\vskip 3mm
\nl Integrate the $2 \times 2$ ODE $\dot{S}(t) = U_1^\top\, F \left( t, U_1 S(t) U_1^\top \right)$,\;
\qquad with initial value $S(t_0) = U_1^\top X_0 U_1 = \left( U_1^\top U_0 \right) X_0 \left( U_1^\top U_0 \right)^\top$.\;
\vskip 3mm
\nl Set $S_1 = S(t_1)/\| S(t_1) \|$,(normalization).\;
\vskip 2mm
\nl Return $S_1, U_1$\;
}
\caption{Low rank symmetry/skew-symmetry preserving integrator}
\label{alg_lr}
\end{algorithm}
Using the example of Section \ref{sec:illex}, integrating Equations \eqref{ode-su1}, \eqref{ode-su2}, and \eqref{ode-su3},
 we obtain the same distance $\eps^*$ and a common null vector $c$ of the same accuracy as when integrating \eqref{eq:ode-o}, i.e.
\small
\begin{equation*}
\eps^* S_1 = \left[ \begin{array}{rr}
   -0.2722  &  0 \\
    0       &  0.0719
\end{array} \right], \
\eps^* S_2 = \left[ \begin{array}{rr}
   -0.2053  &  0 \\
    0       &  0.0205
\end{array} \right], \
\eps^* S_3 = \left[ \begin{array}{rr}
    0       &  0.0529 \\
   -0.0529  &  0
\end{array} \right]
\end{equation*}
and
\begin{equation*}
    U_1 = \left[ \begin{array}{rr}
   -0.2291  &  0.0737 \\
    0.5521  & -0.3270 \\
    0.0449  &  0.0453 \\
    0.0682  &  0.9264 \\
    0.7975  &  0.1658
\end{array} \right], \
U_2 = \left[ \begin{array}{rr}
   -0.3685  & -0.3829 \\
    0.4904  & -0.5782 \\
    0.0689  &  0.1841 \\
   -0.5685  & -0.5870 \\
    0.5439  & -0.3750
\end{array} \right], \
U_3 = \left[ \begin{array}{rr}
   -0.8360  &  0.2633 \\
   -0.1780  & -0.6387 \\
    0.4867  & -0.0431 \\
    0.1176  &  0.3495 \\
   -0.1367  & -0.6315
\end{array} \right].
\end{equation*}
\normalsize
%

\section{Gradient flow for even state dimension}\label{sec:even}
The derivation of the gradients in the case that the space dimension is even is more complicated, since in this case the skew-symmetric matrix $J$ is not guaranteed
to have a zero eigenvalue.
\subsection{Computation of the gradient of the functional \eqref{Feven}}\label{sec:gradneven}

Similarly to the odd case we have
\begin{eqnarray*}
\frac12 \frac d{dt} \mu^2 & = & {}-\eps\,\mu\, \langle \Im \left( w w^H \right), \dot \Gamma \rangle = \eps\,\mu\, \langle \left( \Re (w) \,\Im (w)^\top - \Im (w)\,\Re(w)^\top \right), \dot \Gamma \rangle.
\end{eqnarray*}
Considering orthogonal projections with respect to the Frobenius inner product onto the respective matrix manifolds
$\Sym^{n,n}$, $\Skew^{n,n}$, we identify the constrained gradient directions of the terms associated to eigenvalues as
\begin{equation}
\dot \Delta \propto \l\,   x x^\top, \qquad
\dot \Theta \propto \nu\,  u u^\top, \qquad
\dot \Gamma \propto \mu\, \left( \Re (w)\, \Im (w)^\top - \Im (w)\, \Re (w)^\top  \right).
\end{equation}
Different to the odd case we have to consider
\begin{eqnarray*}
\frac12 \frac d{dt} \left(| x^\top \Re(w) |^2 \right) & = &
\frac12 \frac d{dt} \left( x^\top \Re(w) \Re(w)^\top x \right)
\\
& = &
x^\top \Re(w) \Re(w)^\top \dot x  + \Re(w)^\top x x^\top  \Re(\dot w),
\end{eqnarray*}
and thus
\begin{eqnarray}
\frac12 \frac d{dt} \left( {}-| x^\top \Re(w) |^2 \right)
& = & \eps \left(
(x^\top \Re(w)) \Re(w)^\top G \dot \Delta x  +
(\Re(w)^\top x) x^\top \left( \frac{P \dot \Gamma w + \conj{P} \dot \Gamma \conj{w}}{2}  \right)
\right)
\nonumber
\\
& = & \eps \left(
\Big\langle 
\eta\, G^\top \Re(w) x^\top, \dot \Delta \Big\rangle +
\Big\langle 
\eta\, \Re(P^H x w^H), \dot \Gamma \Big\rangle
\right),
\label{eq:xtRew}
\end{eqnarray}
where $\eta = x^\top \Re(w)$, and $P$ is the pseudoinverse of $J+\eps \Gamma - \iu \mu I$.
Analogously
\begin{eqnarray*}
\frac12 \frac d{dt} \left( {}-| x^\top \Im(w) |^2 \right)
& = & \eps \left(
\Big\langle 
\zeta\, G^\top \Im(w) x^\top, \dot \Delta \Big\rangle +
\Big\langle 
\zeta\, \Im(P^H x w^H), \dot \Gamma \Big\rangle
\right),
\label{eq:xtImw}
\end{eqnarray*}
where $\zeta = x^\top \Im(w)$. Introduce $w_\r := \Re(w)$ and $ w_\i := \Im(w)$.
In order to compute the steepest descent direction, we minimize the gradient of $F_\eps$ and collect the
summands involving $\dot \Delta$, $\dot \Theta$ and those involving $\dot \Gamma$. This yields
\begin{equation*}
\frac d{dt} F_\eps(\Delta,\Theta) =
\eps\,\langle \left(\lambda x + p \right)\,x^\top, \dot \Delta \rangle
+
\eps\,\langle \left( \nu u + q \right)\,u^\top, \dot \Theta \rangle
+
\eps\,\langle W + \eta\,\Re(H) + \zeta\,\Im(H) , \dot \Gamma \rangle,
\end{equation*}
with
\begin{eqnarray*}
p & = & 
\theta G^\top u \ + 
G^\top \left( \eta w_\r + \zeta w_\i \right),
\\
q & = & 
\theta N^\top x,
\\
W & = & w_\r w_\i^\top - w_\i w_\r^\top,
\\
H & = & P^H x w^H.
\end{eqnarray*}
Taking into consideration projection with respect to the Frobenius inner product of the vector field onto the manifolds of symmetric and skew-symmetric matrices, this leads to the system of differential equations,
\begin{eqnarray}
\dot \Delta & = & - \Sym  \left( \left(\lambda x  + p \right)\,x^\top \right)  + \varrho \Delta, \nonumber
\\[2mm]
\dot \Theta & = & - \Sym  \left( \left(\nu  u  + q \right)\,u^\top \right)  + \varrho \Theta,
\label{eq:ode-e}
\\[2mm]
\dot \Gamma & = & - \Skew \left( W + \eta\,\Re(H) + \zeta\,\Im(H) \right)  + \varrho \Gamma, \nonumber
\end{eqnarray}
%
where
\begin{eqnarray*}
\varrho &=& \left( \langle \Delta,  \Sym  \left( (\lambda x  + p)\,x^\top \right) \rangle +
                     \langle \Theta,  \Sym  \left( (\nu  u  + q)\,u^\top \right) \rangle \right .\\ &+&
					         \left .  \langle \Gamma,  \Skew \left(  W + \eta\,\Re(H) + \zeta\,\Im(H) \right) \rangle \right)
\end{eqnarray*}
is again used to ensure the norm conservation. In this way we have again obtained a structured flow with matrices in
$\Sym^{n,n}$ and $\Skew^{n,n}$, respectively.

\begin{remark}
Similarly to the odd case it is possible to derive a rank two gradient system and obtain a more effective
numerical integration.
\end{remark}


\section{A unifying functional}\label{sec:unified}
One may also try to construct a unifying function that treats the odd and even dimension case together.
For this we denote by $x$ the eigenvector associated to $\lambda$, the smallest eigenvalue of $E+\eps \Delta$, with the goal to make this the common null vector in the end. Introduce the alternative functional
\begin{eqnarray}
\widetilde F_\eps(\Delta,\Theta,\Gamma)  & = &
\frac12 \Big(\lambda^2 + \nu^2 + \| \left(R + \eps \Theta \right) x \|_2^2 + \| \left(J + \eps \Gamma \right) x \|_2^2 \Big)
\nonumber
\\
& = & \frac12 \Big(\lambda^2 + \nu^2 +
x^\top \left(R + \eps \Theta \right)^2 x - x^\top \left(J + \eps \Gamma \right)^2 x  \Big),
\label{Fodd2}
\end{eqnarray}
with $\| \left( \Delta, \Gamma, \Theta \right) \|_F = 1$.

We observe that
\begin{eqnarray*}
\frac{1}{2 \eps} \frac d{dt} \left( x^\top \left(R + \eps \Theta \right)^2 x \right) & = &
-x^\top \left(R + \eps \Theta \right)^2 G \dot{\Delta} x + x^\top \left(R + \eps \Theta \right) \dot{\Theta} x
\\
& = &
{}-\langle G \left(R + \eps \Theta \right)^2 x x^\top, \dot{\Delta} \rangle + \langle \left(R + \eps \Theta \right) x x^\top, \dot{\Theta} \rangle,
\end{eqnarray*}
and similarly,
\begin{eqnarray*}
\frac{1}{2 \eps} \frac d{dt} \left( x^\top \left(J + \eps \Gamma \right)^2 x \right) & = &
{}-x^\top \left(J + \eps \Gamma \right)^2 G \dot{\Delta} x + x^\top \left(J + \eps \Gamma \right) \dot{\Gamma} x
\\
& = &
{}-\langle G \left(J + \eps \Gamma \right)^2 x x^\top, \dot{\Delta} \rangle + \langle \left(J + \eps \Gamma \right) x x^\top, \dot{\Gamma} \rangle.
\end{eqnarray*}
This leads to the system of ODEs
\begin{eqnarray}
\dot \Delta & = & -  \Sym\left( \left(\lambda x  + s \right)\,x^\top \right)  + \varrho \Delta, \nonumber
\\[2mm]
\dot \Theta & = & -  \Sym\left( \nu  u \,u^\top  +  t \, x^\top \right) + \varrho \Theta,
\label{eq:ode-oo}
\\[2mm]
\dot \Gamma & = &    \Skew \left( z \,x^\top \right)  + \varrho \Gamma, \nonumber
\end{eqnarray}
where
\[
\varrho = \left( \langle \Delta,  \Sym  \left( \left(\lambda x  + s \right)\,x^\top \right) \rangle +
                     \langle \Theta,  \left( \nu  u \,u^\top  + \Sym\left( t \, x^\top \right)  \right) \rangle -
					           \langle \Gamma,  \Skew \left(  z \,x^\top \right) \rangle \right)
\]
and
\begin{eqnarray}
\nonumber
s & = & G \Big( \left(J + \eps \Gamma \right)^2 - \left(R + \eps \Theta \right)^2 \Big) x,
\\
\label{eq:sto}
t & = & \left(R + \eps \Theta \right) x,
\\
z & = & \left(J + \eps \Gamma \right) x.
\nonumber
\end{eqnarray}
Although this functional appears simpler to manage and does not require the computation of two pseudo-inverses,
our experiments seem to indicate that with the previously considered functionals a higher accuracy can be reached.

\section{The outer iteration for  $\eps$}\label{sec:eps}
In this section we discuss the outer iteration  to compute the optimizing $\eps$. The simplest way to do this is by means of a bisection technique.
\IncMargin{1em}
\begin{algorithm}
\DontPrintSemicolon
\KwData{Matrices $E,R,J$, $k_{\max}$ (max number of iterations), $\delta$, and tolerance tol\;
$\eps_0$, $\eps_{\rm lb}$ and $\eps_{\rm ub}$ (starting values for the lower and upper bounds for $\epstar$)}
\KwResult{$\eps_\delta$ (upper bound for the distance), $\Delta(\epstar), \Theta(\epstar), \Gamma(\epstar)$}
\Begin{
\nl Compute $\Delta(\eps_0), \Theta(\eps_0), \Gamma(\eps_0)$\;
\nl Compute $g(\eps_0)$\;
\nl Set $k=0$\;
\While{$k \le 1$ or $|\eps_{\rm ub} - \eps_{\rm lb}| > {\rm tol}$}{
\nl \eIf{$f({\eps_k}) < {\rm tol}$}
{
Set $ \eps_{\rm ub} = \min(\eps_{\rm ub},\eps_k)$\;
}
{
Set $ \eps_{\rm lb} = \max(\eps_{\rm lb},\eps_k)$\;
}
\nl Compute $\eps_{k+1} = (\eps_{\rm lb} + \eps_{\rm ub})/2$ \ (bisection step)\;\
\eIf{$k=k_{\max}$}
{Return interval $[\eps_{\rm lb},\eps_{\rm ub}]$\; Halt}
{Set $k=k+1$}
\nl Compute $\Delta(\eps_k), \Theta(\eps_k), \Gamma(\eps_k)$\;
\nl Compute $f(\eps_k)$\;
}
\nl Return $\epstar = \eps_k$\;
}
\caption{Bisection method for distance approximation}
\label{alg_dist}
\end{algorithm}
\subsection{An illustrative example} \label{sec:mks}
To illustrate the performance of the described algorithm, we consider a scalable linear mass-spring-damper system, that has been used as a model reduction test case in \cite{HMM19}.
\begin{figure}[t]
\centerline{
\includegraphics[scale=0.5]{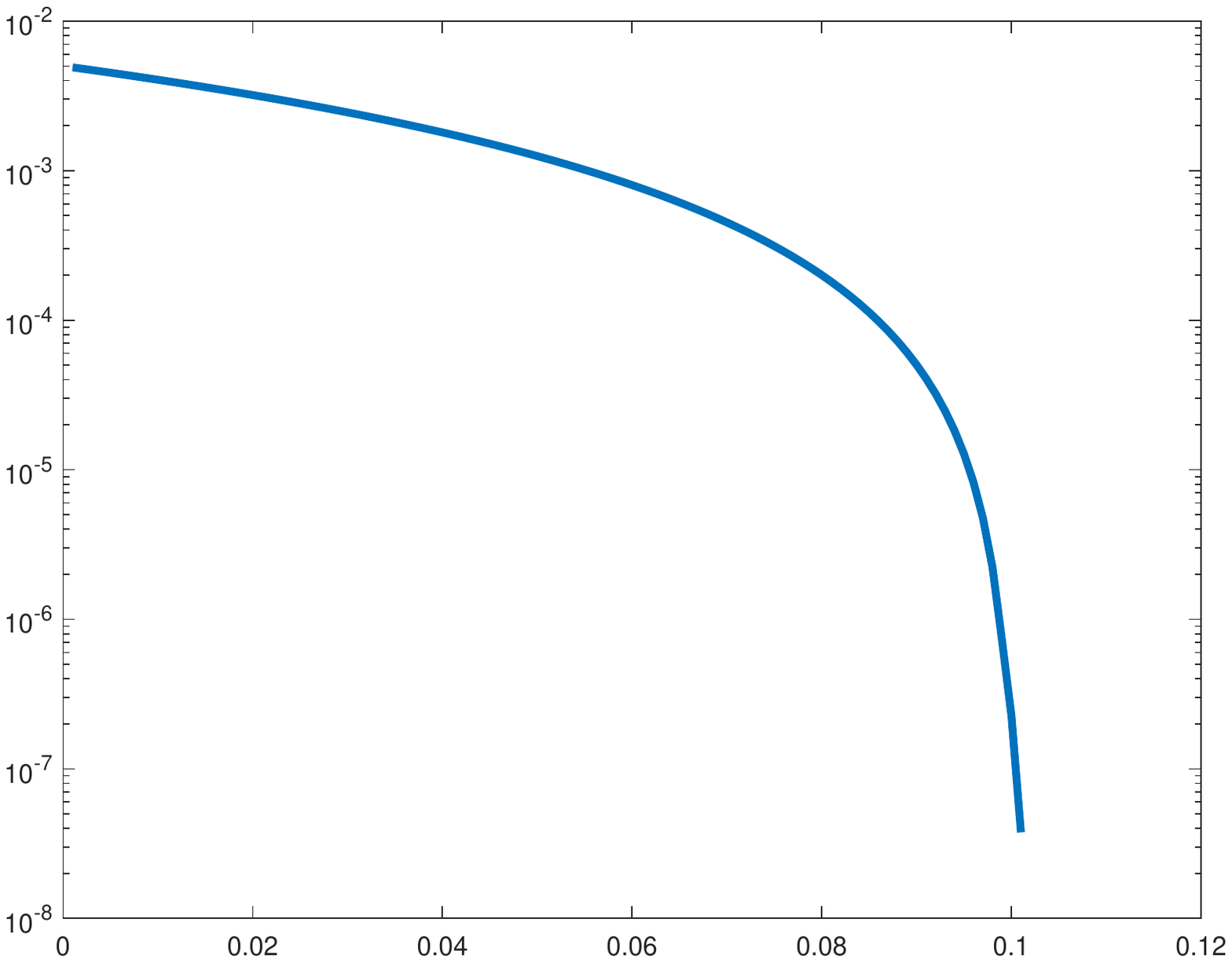}
}
\vspace{-3.6cm}
\caption{The function $f(\eps)$ for the mass-spring-damper example. \label{fig:exvm}}
\end{figure}

This test case generates matrices $M=M^T\geq 0,G=-G^T,D=D^T\geq0,K=K^T> 0$ and, after first order formulation, leads to a dH pencil
\[
\lambda \underbrace{\begin{bmatrix} K & 0 & 0\\ 0 & M & 0\\ 0&0&0\end{bmatrix}}_E-\left(\underbrace{\begin{bmatrix}0&K&0\\-K&0&-G^T\\
0&G&0\end{bmatrix}}_J-\underbrace{\begin{bmatrix}0&0&0\\0&D&0\\0&0&0\end{bmatrix}}_R\right)
\]
This pencil is regular and of index two. If one puts
$\gamma I$ in the $(3,3)$-block of  the matrix $R$, then the distance to index $2$ and instability is $\gamma$. Choosing the dimension $N=100$ we obtain matrices $E,R,J \in \R^{3 N +1}$.

We fix for $\gamma=10^{-1}$.
The plot of the function $f(\eps)$ obtained by integrating \eqref{eq:ode-o2} for increasing $\eps$,  is given in Figure \ref{fig:exvm}.

\section*{Conclusions and further work}

We have investigated a structured distance problem related to the study of port-Hamiltonian systems, that is determining
the closest triplet of matrices to a given one, sharing a common null-space.
%
%


\section*{Acknowledgments}

The first author acknowledges that his research was supported by funds from the
Italian MUR (Ministero dell'Universit\`a e della Ricerca) within the PRIN 2017
Project ``Discontinuous dynamical systems: theory, numerics and applications''
and by the INdAM Research group GNCS (Gruppo Nazionale di Calcolo Scientifico).

The second author  thanks Deutsche Forschungsgemeinschaft (DFG)
for  support within the project B03 in CRC TRR 154.


%
%
%


\begin{thebibliography}{10}
\providecommand{\url}[1]{#1}
\csname url@samestyle\endcsname
\providecommand{\newblock}{\relax}
\providecommand{\bibinfo}[2]{#2}
\providecommand{\BIBentrySTDinterwordspacing}{\spaceskip=0pt\relax}
\providecommand{\BIBentryALTinterwordstretchfactor}{4}
\providecommand{\BIBentryALTinterwordspacing}{\spaceskip=\fontdimen2\font plus
\BIBentryALTinterwordstretchfactor\fontdimen3\font minus
\fontdimen4\font\relax}
\providecommand{\BIBdecl}{\relax}
\BIBdecl

\bibitem{AchAM21}
F. Achleitner, A. Arnold, and V. Mehrmann.
Hypocoercivity and controllability in linear semi-dissipative
ODEs and DAEs. http://arxiv.org/abs/2104.07619
Submitted for publication, 2021. 
\bibitem{AliBMSV17}
N. Aliyev, P. Benner, E. Mengi, P. Schwerdtner, and M. Voigt.
\newblock Large-scale computation of $L_\infty$-norms by a greedy subspace method.
{\em SIAM J. Matrix Anal. Appl.}, 38, 1496--1516, 2017.

\bibitem{AliMM20}
N.~Aliyev, V.~Mehrmann, and E.~Mengi.
\newblock Computation of stability radii for large-scale dissipative
  Hamiltonian systems.
\newblock {\em Advances Comp. Math.}, 46:6,  2020.

\bibitem{AltMU20}
R. Altmann, V. Mehrmann, and B. Unger.
Port-Hamiltonian formulations of poroelastic network models.
http://arxiv.org/abs/2012.01949.
Submitted for publication, 2020.

\bibitem{BeaMXZ18}
C.~Beattie, V.~Mehrmann, H.~Xu, and H.~Zwart.
\newblock Port-{H}amiltonian descriptor systems.
\newblock {\em Math. Control, Signals, Sys.}, 30:17, 2018.
\newblock https://doi.org/10.1007/s00498-018-0223-3.

\bibitem{BeaMV19}
C.~Beattie, V.~Mehrmann, and P.~{Van Dooren}.
\newblock Robust port-{H}amiltonian representations of passive systems.
\newblock {\em Automatica}, 100:182--186, 2019.

\bibitem{BerGTWW17}
T.~Berger, H.~Gernandt, C.~Trunk, H.~Winkler, and M.~Wojtylak.
\newblock A new bound for the distance to singularity of a regular matrix
  pencil.
\newblock In {\em Proc. Appl. Mathematics and Mechanics}, volume 17.1, pages
  863--864. Wiley Online Library, 2017.

\bibitem{BerGTWW19}
T.~Berger, H.~Gernandt, C.~Trunk, H.~Winkler, and M.~Wojtylak.
\newblock The gap distance to the set of singular matrix pencils.
\newblock {\em Linear Alg. Appl.}, 564:28--57, 2019.

\bibitem{BreCP96}
K.~E. Brenan, S.~L. Campbell, and L.~R. Petzold.
\newblock {\em Numerical Solution of Initial-Value Problems in Differential
  Algebraic Equations}.
\newblock {SIAM} Publications, Philadelphia, PA, 2nd edition, 1996.

\bibitem{ByeHM98}
R.~Byers, C.~He, and V.~Mehrmann.
\newblock Where is the nearest non-regular pencil.
\newblock {\em Linear Algebra Appl.}, 285:81--105, 1998.

\bibitem{CL20}
G.~Ceruti and C.~Lubich.
\newblock Time integration of symmetric and anti-symmetric low-rank matrices and Tucker tensors.
\newblock {\em BIT Numerical Mathematics}, 60:591--614, 2020.


\bibitem{Dai89}
L.~Dai.
\newblock {\em Singular Control Systems}.
\newblock Number 118 in Lecture Notes in Control and Information Sciences.
  Springer-Verlag, Berlin, 1989.

\bibitem{DuLM13}
N.H. Du, V.H. Linh and V. Mehrmann, \newblock
{\em Robust stability of differential-algebraic equations}.
 In \emph{Differential Algebraic Equation Forum.}
 Surveys in Differential-Algebraic Equations I,
 A. Ilchmann and T. Reis Edtrs. pp. 63--96, 2013.

\bibitem{EggKLMM18}
H. Egger, T. Kugler, B. Liljegren-Sailer, N. Marheineke, and V. Mehrmann,
\newblock On structure preserving model reduction for damped wave propagation in transport networks.
\newblock {\em SIAM J. Sci. Comp.}, 40:A331--A365, 2018.

\bibitem{Fre08}
R.~W. Freund.
\newblock Structure-preserving model order reduction of rcl circuit equations.
\newblock In {\em Model {O}rder {R}eduction: {T}heory, {R}esearch {A}spects and
  {A}pplications}, pages 49--73. Springer, 2008.

\bibitem{Fre11}
R.~W. Freund.
\newblock The {SPRIM} algorithm for structure-preserving order reduction of
  general rcl circuits.
\newblock In {\em Model reduction for circuit simulation}, pages 25--52.
  Springer, 2011.

\bibitem{Gan59a}
F.~R. Gantmacher.
\newblock {\em Theory of Matrices}, volume~1.
\newblock Chelsea, New York, 1959.

\bibitem{GilMS18}
N.~Gillis, V.~Mehrmann, and P.~Sharma.
\newblock Computing nearest stable matrix pairs.
\newblock {\em Numer. Lin. Alg. Appl.}, 25:e2153, 2018.

\bibitem{GilS17}
N.~Gillis and P.~Sharma.
\newblock On computing the distance to stability for matrices using linear
  dissipative hamiltonian systems.
\newblock {\em Automatica}, 85:113--121, 2017.

\bibitem{GilS18}
N.~Gillis and P.~Sharma.
\newblock Finding the nearest positive-real system.
\newblock {\em {SIAM} J. Matrix Anal. Appl.}, 56(2):1022--1047, 2018.

\bibitem{GramQSW16}
N.~Gr{\"a}bner, V.~Mehrmann, S.~Quraishi, C.~Schr\"oder, and U.~{von W}agner.
\newblock Numerical methods for parametric model reduction in the simulation of
  disc brake squeal.
\newblock {\em Z. Angew. Math. Mech.}, 96:1388--1405, 2016.

\bibitem{GugLM17}
N.~Guglielmi, C.~Lubich, and Volker Mehrmann.
\newblock On the nearest singular matrix pencil.
\newblock {\em {SIAM} J. Matrix Anal. Appl.}, 38(3):776--806, 2017.

\bibitem{HMM19}
S-A.~Hauschild, N.~Marheineke, and Volker Mehrmann.
\newblock Model reduction techniques for port‐Hamiltonian differential‐algebraic systems.
\newblock {\em Control and Cybernetics}, 48(1):1--19, 2019.

\bibitem{JacZ12}
B.~Jacob and H.~Zwart.
\newblock {\em Linear port-{H}amiltonian systems on infinite-dimensional
  spaces}.
\newblock Operator Theory: Advances and Applications, 223.
  Birkh{\"a}user/Springer Basel AG, Basel CH, 2012.

\bibitem{Kat95}
T.~Kato.
\newblock \emph{Perturbation theory for linear operators}.\hskip 1em plus 0.5em
  minus 0.4em\relax Springer-Verlag, 1995.
	
\bibitem{KL08}
O.~Koch and Ch.~Lubich.
\newblock Dynamical low-rank approximation.
\newblock {\em SIAM J. Matrix Anal. Appl.}, 29(2):434--454, 2007.	
	
\bibitem{KunM06}
P.~Kunkel and V.~Mehrmann.
\newblock {\em Differential-Algebraic Equations. Analysis and Numerical
  Solution}.
\newblock EMS Publishing House, Z{\"u}rich, Switzerland, 2006.

\bibitem{MehMS16}
C.~Mehl, V.~Mehrmann, and P.~Sharma.
\newblock Stability radii for linear hamiltonian systems with dissipation under
  structure-preserving perturbations.
\newblock {\em SIAM J. Matrix Anal. Appl.}, 37:1625--1654, 2016.

\bibitem{MehMS17}
C.~Mehl, V.~Mehrmann, and P.~Sharma.
\newblock Stability radii for real linear {H}amiltonian systems with perturbed
  dissipation.
\newblock {\em BIT Numerical Mathematics}, 57:811--843, 2017.

\bibitem{MehMW15}
C.~Mehl, V.~Mehrmann, and M.~Wojtylak.
\newblock On the distance to singularity via low rank perturbations.
\newblock {\em Operators and Matrices}, 9:733--772, 2015.

\bibitem{MehMW18}
C.~Mehl, V.~Mehrmann, and M.~Wojtylak.
\newblock Linear algebra properties of dissipative {H}amiltonian descriptor
  systems.
\newblock {\em {SIAM} J. Matrix Anal. Appl.}, 39(3):1489--1519, 2018.

\bibitem{MehMW20} 
C. Mehl, V. Mehrmann, and M. Wojtylak,
\newblock Distance problems for dissipative Hamiltonian systems and related matrix polynomials.
http://arxiv.org/abs/2001.08902.
\newblock {Linear Alg. Appl.}, https://doi.org/10.1016/j.laa.2020.05.026, 2020.

\bibitem{MehM19}
V.~Mehrmann and R.~Morandin.
\newblock Structure-preserving discretization for port-hamiltonian descriptor
  systems.
\newblock In {\em 58th IEEE Conf. Decision and Control (CDC), Nice},
  pages 6863--6868, 2019.
\newblock https://arXiv:1903.10451.

\bibitem{MehV20}
V. Mehrmann  and P. Van Dooren,
\newblock Optimal robustness of port-Hamiltonian systems,
\newblock {\em SIAM J. Matrix Anal. Appl.}, 41:134--151, 2020.



\bibitem{MeyS88}
C.D. Meyer and G.W. Stewart.
\newblock Derivatives and perturbations of eigenvectors.
\newblock {\em {SIAM} J. Numer. Anal.}, 25:679--691, 1988.
	

\bibitem{hanso} M. L. Overton. \newblock HANSO: Hybrid Algorithm for Non-Smooth Optimization.
\newblock http://www.cs.nyu.edu/overton/software/hanso. Date 23.11.20

\bibitem{Sch13}
A.~J.~{van der} Schaft.
\newblock Port-{H}amiltonian differential-algebraic systems.
\newblock In {\em Surveys in Differential-Algebraic Equations I}, pages
  173--226. Springer-Verlag, 2013.

\bibitem{SchM02}
A.~J.~{van der} Schaft and B.~M. Maschke.
\newblock {H}amiltonian formulation of distributed-parameter systems with
  boundary energy flow.
\newblock {\em J. Geom. Phys.}, 42:166--194, 2002.

\bibitem{SchM18}
A.~van~der Schaft and B.~Maschke.
\newblock Generalized port-{H}amiltonian dae systems.
\newblock {\em Systems \& Control Letters}, 121:31--37, 2018.

\bibitem{SchJ14}
A.~J. {van der}~Schaft and D.~Jeltsema.
\newblock Port-{H}amiltonian systems theory: An introductory overview.
\newblock {\em Found. and Trends in Systems and Control}, 1(2-3):173--378,
  2014.

\bibitem{W84}
W.C.~Waterhouse.
\newblock The codimension of singular matrix pairs.
\newblock {\em Linear Algebra Appl.}, 57:227--245, 1984.

\bibitem{W50}
M.A.~Woodbury
\newblock {\em Inverting modified matrices}.
\newblock Statistical Research Group, Memo. Rep. no. 42,
\newblock Princeton University, Princeton, N.J., 4 pp, 1950.
\end{thebibliography}
\end{document}